\documentclass[12pt]{amsart} 
\usepackage{amssymb}
\usepackage{comment}
\usepackage{amsmath}
\usepackage{enumitem}
\usepackage[pdftex]{hyperref}
\usepackage{mathtools}
\usepackage{wrapfig}
\usepackage{tikz}
\usetikzlibrary{intersections}
\usetikzlibrary{calc}
\usetikzlibrary{arrows.meta}
\theoremstyle{plain} 
\newtheorem{theorem}{\indent\sc Theorem}[section]
\newtheorem{lemma}[theorem]{\indent\sc Lemma}
\newtheorem{corollary}[theorem]{\indent\sc Corollary}
\newtheorem{proposition}[theorem]{\indent\sc Proposition}

\theoremstyle{definition} 
\newtheorem{definition}[theorem]{\indent\sc Definition}
\newtheorem{remark}[theorem]{\indent\sc Remark}
\newtheorem{example}[theorem]{\indent\sc Example}
\newtheorem{notation}[theorem]{\indent\sc Notation}

\newcommand{\R}{\mathbb{R}}

\newcommand{\Z}{\mathbb{Z}}

\newcommand{\norm}[1]{\left\lVert#1\right\rVert}
\newcommand{\Cay}{\Gamma}

\usepackage{color}	
\begin{document}

\title[Free products of coarsely convex spaces]
{Free products of coarsely convex spaces and the coarse Baum--Connes conjecture} 

\author[Tomohiro Fukaya]{Tomohiro Fukaya} 

\author[Takumi Matsuka]{Takumi Matsuka} 



\renewcommand{\thefootnote}{\fnsymbol{footnote}}
\footnote[0]{2020\textit{ Mathematics Subject Classification}.
Primary 51F30; Secondary 20F65, 58B34.}

\keywords{ 
Free products, coarsely convex spaces, the coarse Baum--Connes conjecture.
}
\thanks{ 
The first author was supported by JSPS KAKENHI Grant number JP19K03471. 
The second author was supported by JST, the establishment of university fellowships towards the creation of science technology innovations, Grant number JPMJFS2139.
}
\address{Tomohiro Fukaya \endgraf
Department of Mathematical Sciences \endgraf
Tokyo Metropolitan University \endgraf
Minami-osawa Hachioji Tokyo 192-0397 \endgraf
Japan
}
\email{tmhr@tmu.ac.jp}

\address{Takumi Matsuka \endgraf
Department of Mathematical Sciences \endgraf
Tokyo Metropolitan University \endgraf
Minami-osawa Hachioji Tokyo 192-0397 \endgraf
Japan
}
\email{takumi.matsuka1@gmail.com}


\maketitle

\begin{abstract}
The first author and Oguni introduced a wide class of metric spaces, called coarsely convex spaces.
It includes Gromov hyperbolic metric spaces, CAT(0) spaces, systolic complexes, proper injective metric spaces.
We introduce the notion of free products of metric spaces
and show that free products of symmetric geodesic coarsely convex spaces are also symmetric geodesic coarsely convex spaces.
As an application, it follows that free products of symmetric geodesic coarsely convex spaces satisfy the coarse Baum--Connes conjecture.
\end{abstract}

\section*{Introduction} 
The first author and Oguni \cite{FO} introduced a wide class of metric spaces, called \textit{coarsely convex spaces}.
Coarsely convex spaces can be regarded as the counterpart of simply connected, complete, Riemannian manifolds with non-positive sectional curvature in coarse geometry.
This class includes many classes of metric spaces.
In particular, a class of \textit{symmetric geodesic coarsely convex spaces} is an important subclass of coarsely convex spaces.
\par Let $(X,d_X)$ be a geodesic metric space.
Let $E\geq 1$ and $C \geq 0$ be constants.
Let $\mathcal{L}$ be a family of geodesic segments.
The metric space $X$ is a symmetric geodesic $(E,C,\mathcal{L})$-coarsely convex space, if $E$, $C$, and $\mathcal{L}$ satisfy the following conditions: 
\begin{enumerate}
\renewcommand{\labelenumi}{\roman{enumi}).}
\item For all $u,v \in X$, there exists a geodesic segment $\gamma \in \mathcal{L}$ with $\gamma : [0,a] \to X$ such that $\gamma(0)=u$ and $\gamma(a)=v$.
\item Let $\gamma, \eta \in \mathcal{L}$ be geodesic segments with $\gamma : [0,a] \to X$ and $\eta : [0,b] \to X$.
        Let $\gamma(0)=\eta(0)$.
        For all $t \in [0,a]$, $s \in [0,b]$, and $c\in [0,1]$, we have
\begin{equation*}
d_X(\gamma(ct),\eta(cs)) \leq cEd_X(\gamma(t),\eta(s)) + C.
\end{equation*}
\item For a map $\gamma : [a,b] \to X$, we denote by $\gamma^{-1}$, the map $\gamma^{-1}:[a,b] \to X$ defined by $\gamma^{-1}(t)\coloneqq\gamma(b-(t-a))$ for $ t \in [a,b]$.
        The geodesic segment $\gamma^{-1}$ is in $\mathcal{L}$ for each $\gamma \in \mathcal{L}$.
\end{enumerate}
We say that $X$ is a \textit{symmetric geodesic coarsely convex space} if there exist $E$, $C$, and a family of geodesic segments $\mathcal{L}$
such that $X$ is a symmetric geodesic $(E,C,\mathcal{L})$-coarsely convex space.
A \textit{coarsely convex group} is a group $G$ acting properly and cocompactly on a coarsely convex space.
The classes of metric spaces listed in Table \ref{Tab1} are examples of symmetric geodesic coarsely convex spaces.
\begin{table}[h]
\begin{itemize}
\item Geodesic Gromov hyperbolic metric spaces.
\item CAT(0) spaces.
\item Systolic complexes \cite{OP}.
\item Proper injective metric spaces \cite{DL}, especially, the
      injective hulls of locally finite coarsely Helly graphs \cite{Helly}.
\end{itemize}
\caption{Examples of symmetric geodesic coarsely convex spaces.}
\label{Tab1}
\end{table}
\par Pisanski and Tucker \cite{PT} introduced free products of Cayley graphs,
and Bridson and Haefliger \cite[Theorem II.11.18]{BH} constructed metric spaces on which free products of groups act properly and cocompactly.
By slightly modifying their construction, we define \textit{free products of metric spaces with nets}, 
and we obtain the following result.
\begin{theorem}\label{IFM}
Let $X$ and $Y$ be metric spaces with nets.
If $X$ and $Y$ are symmetric geodesic coarsely convex spaces,
then the free product $X*Y$ is a symmetric geodesic coarsely convex space.
\end{theorem}
\par Let $X$ be a proper metric space.
The \textit{coarse assembly map} is a homomorphism from the \textit{coarse K-homology} of $X$ to the $K$-theory of the \textit{Roe algebra} of $X$.
The \textit{coarse Baum--Connes conjecture} states that for ``nice'' proper metric spaces, the coarse assembly maps are isomorphisms.
The first author and Oguni \cite[Theorem 1.3]{FO} showed that for proper coarsely convex spaces, the coarse Baum--Connes conjecture holds.
Combining this result and Theorem \ref{IFM}, we obtain the following
\begin{theorem}\label{IFM2}
Let $X$ and $Y$ be proper metric spaces with nets.
If $X$ and $Y$ are symmetric geodesic coarsely convex spaces,
then the free product $X*Y$ satisfies the coarse Baum--Connes conjecture.
\end{theorem}
Let $G$ and $H$ be groups acting properly and cocompactly on $X$ and $Y$, respectively.
We can construct the free product $X$ and $Y$ with respect to their actions.
Moreover, $G*H$ acts properly and cocompactly on it.
Therefore, combining Theorem \ref{IFM2} and the \v{S}varc--Milnor Lemma, we obtain the following.
\begin{theorem}
Let $X$ and $Y$ be proper metric spaces with nets.
We suppose that $X$ and $Y$ are symmetric geodesic coarsely convex spaces.
Let $G$ and $H$ be groups acting properly and cocompactly on $X$ and $Y$, respectively.
Then $G*H$ satisfies the coarse Baum--Connes conjecture.
\end{theorem}


\section{Free products of metric spaces}
To prepare for the definition of free products of metric spaces, we introduce several notations.

\begin{definition}
\label{def:net}
Let $(X,d_X)$ be a metric space and let $X_0$ be a set.
An \textit{index map} is a map $i_X: X_0 \to X$ such that for any compact subset $K \subset X$, the preimage $i_X^{-1}(K)$ is a finite set.
We choose a \textit{base point} $e_X\in i_X(X_0)$.
We call $(X_0, i_X,e_X)$ a \textit{net} of $X$.
For  $x_0 \in X_0$, we denote by $\overline{x_0}$ the image $i_X(x_0)$.
\end{definition}
\begin{example}
\label{eg:net}
Let $(X,d_X,e_X)$ be a metric space with a base point $e_X$.
Let $G$ be a group acting on $X$ by isometries.
We say that $G$ \textit{acts properly} on $X$ if for any compact subset  $B \subset X$, the set 
\begin{equation*}
\{ g \in G \mid g(B) \cap B \neq \emptyset \}
\end{equation*}
is a finite set.
We say that $G$ \textit{acts cocompactly} on $X$ if there exists a compact subset $K \subset X$ such that
\begin{equation*}
\bigcup_{g \in G} g(K)=X.
\end{equation*}
When a group $G$ acts properly and cocompactly on $X$, the orbit map $o(e_X) : G \to X, g \mapsto g(e_X)$ is an index map of $X$.
Then $(G, o(e_X),e_X)$ is a net of $X$. 
We call $(G,o(e_X),e_X)$ the \textit{G-net}.
\end{example}

\begin{remark}
 In the definition of the net in Definition~\ref{def:net}, we do not require that
 the index map is injective. Indeed, let $G$ be a group acting on a metric space $X$ 
 with base point $e_X$ as in Example~\ref{eg:net}.
 Then the orbit map $o(e_X)\colon G\to X$ is injective if and only if the stabilizer of 
 $e_X$ is trivial.
\end{remark}

Let $(X,d_X)$ and $(Y,d_Y)$ be metric spaces with nets and
let $(X_0, i_X,e_X)$ and $(Y_0, i_Y,e_Y)$ be nets of $X$ and $Y$, respectively.
We choose $\epsilon_X$ and $\epsilon_Y$ such that $i_X(\epsilon_X)=e_X$ and $i_Y(\epsilon_Y)=e_Y$, respectively.
Let $X^{*}_0=X_0 \setminus \{\epsilon_X\}$ and $Y^{*}_0=Y_0 \setminus \{\epsilon_Y\}$. 
A \textit{normal word} on $X_0^*\sqcup Y_0^*$ is a finite sequence
\begin{align*}
 w_0w_1\cdots w_n \quad (n\geq 0, w_i\in X_0^*\sqcup Y_0^*)
\end{align*}
such that for all $i$, we have
\begin{align*}
 &w_i \in X_0^* \Rightarrow w_{i+1}\in Y_0^*,\\
 &w_i \in Y_0^* \Rightarrow w_{i+1}\in X_0^*.
\end{align*}
Let $\omega = w_0w_1\cdots w_n$ be a normal word. 
We define the \textit{length} of $\omega$ to be $n+1$. We denote by $\epsilon$ 
the \textit{empty word}. We define that the length of the empty word is zero.

We define $W$ to be the set consisting of the empty word $\epsilon$ and all normal words.
We define $W_X$ to be the subset of $W$ consisting of $\epsilon$ and normal words
whose last letter  does not belong to $X_0^*$, that is,
\begin{align*}
 W_X\coloneqq& \{w_0w_1\cdots w_n\in W \mid w_n\notin X_0^*\}\bigsqcup \{\epsilon\},
\end{align*}
and we similarly define $W_Y$, that is,
\begin{align*}
 W_Y\coloneqq& \{w_0w_1\cdots w_n\in W \mid w_n\notin Y_0^*\}\bigsqcup \{\epsilon\}.
\end{align*}
We use the following disjoint union $\widetilde{X*Y}$, 
\begin{align*}
\widetilde{X*Y} \coloneqq & (W_X \times X) \bigsqcup (W_Y \times Y) \\
                     &\bigsqcup (W_X \times X_0^* \times [0,1]) \bigsqcup (W_Y \times Y_0^* \times [0,1]) \bigsqcup (\epsilon \times [0,1]).
\end{align*}
We define an equivalence relation $\sim$ on $\widetilde{X*Y}$ as follows:
\begin{itemize}
\item $(\epsilon, e_X) \sim (\epsilon, 0)$ and $(\epsilon, 1) \sim (\epsilon, e_Y)$. 
\item Let $\omega \in W_X$ and $x_0 \in X_0^*$.
      \begin{align*}
       \{\omega\} \times X \ni (\omega, \overline{x_0}) 
       &\sim 
       (\omega, x_0, 0) \in \{\omega\} \times  \{x_0\} \times [0,1],\\
       \{\omega x_0\} \times Y\ni (\omega x_0, e_Y)
       &\sim
       (\omega, x_0, 1)   \in   \{\omega\} \times  \{x_0\} \times [0,1]. 
      \end{align*}
\item Let $\tau \in W_Y$ and $y_0 \in Y_0^*$.
      we define 
      \begin{align*}
       \{\tau\} \times Y \ni (\tau, \overline{y_0}) 
       &\sim 
         (\tau, y_0, 0) \in \{\tau\} \times  \{y_0\} \times [0,1],\\
       \{\tau y_0\} \times X\ni (\tau y_0, e_X) 
       &\sim
         (\tau, y_0, 1) \in \{\tau\} \times  \{y_0\} \times [0,1].
      \end{align*}
\end{itemize}
\begin{definition}
\textit{The free product $X*Y$} is the quotient of $\widetilde{X*Y}$ 
 by the equivalent relation $\sim$. 
\end{definition}
The free product $X*Y$ consists of the following two types of components (see Figure \ref{fig 1}).
\begin{itemize}
\item The \textit{sheets} consists of 
$\{\omega\}\times X$ and $\{\tau\} \times Y$, where $\omega \in W_X$ and $\tau \in W_Y$.
For simplicity, we write $\{\omega\}\times X$ and $\{\tau\}\times Y$ as $\omega X$ and $\tau Y$, respectively.
We identify $X$ and respectively $Y$ with $\epsilon X$ and respectively 
$\epsilon Y$.
We call $\omega$ and $\tau$ \textit{index words}.
For each sheet, the \textit{height of the sheet} is the length of the index word.
\item The \textit{edges} consists of the following three:
\begin{itemize}
\item There exists an edge $\{\epsilon\}\times [0,1]$
connecting $(\epsilon, e_X) \in \{\epsilon\} \times X$ and $(\epsilon, e_Y) \in \{\epsilon\} \times Y$
\item Let $\omega \in W_X$ and $x_0 \in X_0^*$.
      Then there exists an edge $\{\omega\} \times \{x_0\}\times [0,1]$
      connecting $(\omega, \overline{x_0}) \in \{\omega\} \times X$ and $(\omega x_0, e_Y) \in \{\omega x_0\} \times Y$.
\item Let $\tau \in W_Y$ and $y_0 \in Y_0^*$.
      Then there exists an edge $\{\tau\} \times \{y_0\}\times [0,1]$
      connecting $(\tau, \overline{y_0}) \in \{\tau\} \times Y$ and $(\tau y_0, e_X) \in \{\tau y_0\} \times X$.
\end{itemize}
\end{itemize}
       \begin{figure}[h]
       \begin{tikzpicture}[scale=1]
       \draw(-1,2)--++(4,0)--++(-1,-1)--++(-4,0)--cycle; 
       \draw(9,2)--++(4,0)--++(-1,-1)--++(-4,0)--cycle; 
       \draw(3,2)node[right]{$x_0y_0 X$};
       \draw(13,2)node[right]{$y_0^{\prime}x_1^{\prime} Y$};
       \draw(0,0)--++(4,0)--++(-1,-1)--++(-4,0)--cycle; 
       \draw(8,0)--++(4,0)--++(-1,-1)--++(-4,0)--cycle; 
       \draw(4,0)node[right]{$x_0 Y$};
       \draw(12,0)node[right]{$y_0^{\prime} X$};
       \draw(1,-2)--++(4,0)--++(-1,-1)--++(-4,0)--cycle; 
       \draw(7,-2)--++(4,0)--++(-1,-1)--++(-4,0)--cycle; 
       \draw(3,-2.5)to[out=85,in=105](8,-2.5);
       \draw(2,-2.5)--++(0,1.5);
       \draw[dashed](2,-1.5)--++(0,1);
       \draw(2,-2.5)--++(0.75,0.75);
       \draw[dashed](2.75,-1.75)--++(0.5,0.5);
       \draw(2,-2.5)--++(-0.75,0.75);
       \draw[dashed](1.25,-1.75)--++(-0.5,0.5);
      \draw(2,-2.5)node[below]{$\overline{x_0}$};
      \draw(2,-0.5)node[right]{$e_Y$};
       \draw(1,-0.5)--++(0,1.5);
       \draw[dashed](1,1)--++(0,0.5);
       \draw(1,-0.5)--++(0.75,0.75);
       \draw[dashed](1.75,0.25)--++(0.5,0.5);
       \draw(1,-0.5)--++(-0.75,0.75);
       \draw[dashed](0.25,0.25)--++(-0.5,0.5);
       \draw(0.5,-0.5)node{$\overline{y_0}$};
       \draw(1,1.5)node[left]{$e_X$};
       \draw(1,-2.5)--++(-2,1);
       \draw[dashed](-1,-1.5)--++(-1,0.5);
       \draw(10,-0.5)--++(0,1.5);
       \draw[dashed](10,1)--++(0,0.5);
       \draw(10,-0.5)--++(0.75,0.75);
       \draw[dashed](10.75,0.25)--++(0.5,0.5);
       \draw(10,-0.5)--++(-0.75,0.75);
       \draw[dashed](9.25,0.25)--++(-0.5,0.5);
       \draw(10.5,-0.5)node{$\overline{x_1^{\prime}}$};
       \draw(10,1.5)node[right]{$e_Y$};
       \draw(9,-2.5)--++(0,1.5);
       \draw[dashed](9,-1)--++(0,0.5);
       \draw(9,-2.5)--++(0.75,0.75);
       \draw[dashed](9.75,-1.75)--++(0.5,0.5);
       \draw(9,-2.5)--++(-0.75,0.75);
       \draw[dashed](8.25,-1.75)--++(-0.5,0.5);
       \draw(9.5,-2.5)node{$\overline{y_0^{\prime}}$};
       \draw(9,-0.5)node[left]{$e_X$};
       \draw(-0.5,1.5)--++(0,1.5);
       \draw[dashed](-0.5,3)--++(0,0.5);
       \draw(-0.5,1.5)--++(0.75,0.75);
       \draw[dashed](0.25,2.25)--++(0.5,0.5);
       \draw(-0.5,1.5)--++(-0.75,0.75);
       \draw[dashed](-1.25,2.25)--++(-0.5,0.5);
       \draw(2,1.5)--++(0,1.5);
       \draw[dashed](2,3)--++(0,0.5);
       \draw(2,1.5)--++(0.75,0.75);
       \draw[dashed](2.75,2.25)--++(0.5,0.5);
       \draw(2,1.5)--++(-0.75,0.75);
       \draw[dashed](1.25,2.25)--++(-0.5,0.5);
       \draw(11.5,1.5)--++(0,1.5);
       \draw[dashed](11.5,3)--++(0,0.5);
       \draw(11.5,1.5)--++(0.75,0.75);
       \draw[dashed](12.25,2.25)--++(0.5,0.5);
       \draw(11.5,1.5)--++(-0.75,0.75);
       \draw[dashed](10.75,2.25)--++(-0.5,0.5);
       \draw(9,1.5)--++(0,1.5);
       \draw[dashed](9,3)--++(0,0.5);
       \draw(9,1.5)--++(0.75,0.75);
       \draw[dashed](9.75,2.25)--++(0.5,0.5);
       \draw(9,1.5)--++(-0.75,0.75);
       \draw[dashed](8.25,2.25)--++(-0.5,0.5);
       \draw(3,-2.5)node[right]{$e_X$};
       \draw(8,-2.5)node[left]{$e_Y$};
       \draw(0,-3)node[below]{$\epsilon X$};
       \draw(10,-3)node[below]{$\epsilon Y$};
       \draw(10,-2.5)--++(2,1);
       \draw[dashed](12,-1.5)--++(1,0.5);
       \end{tikzpicture}
       \caption{The free product $X*Y$.
                   Here, $x_0,x_1^{\prime} \in X_0^*$ and $y_0, y_0^{\prime} \in Y_0^*$.}
       \label{fig 1}
       \end{figure}
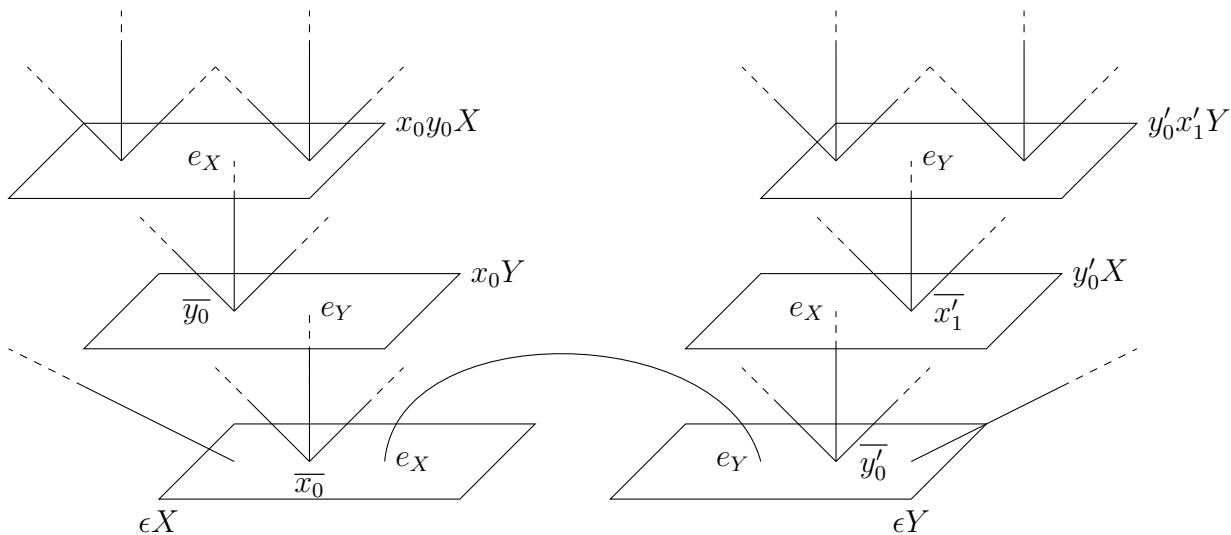
\begin{notation}
\label{nt:coord}
 Each point $p\in X*Y$ can naturally identified with the triplet $(\omega,z,t)$
 where $\omega\in W,\, z\in X \sqcup Y$ and $t\in [0,1)$ satisfying:
 \begin{itemize}
  \item If $\omega \in W_X$, then $z\in X$,
  \item If $\omega \in W_Y$, then $z\in Y$,
  \item If $z\notin i_X(X_0)\sqcup i_Y(Y_0)$, then $t=0$.
 \end{itemize}
 We call $(\omega,z,t)$ the \textit{coordinate} of $p$. 
 We say that $p$ \textit{belongs to sheets} if $t=0$, and
    say that $p$ \textit{belongs to edges} if $t>0$.
 When $p$ belongs to a sheet, we abbreviate $(\omega,z,0)$ as $(\omega,z)$.
\end{notation}

\begin{notation}
 We use the following notations.
 \begin{itemize}
  \item For $z\in X\sqcup Y$, set
        \begin{align*}
         \norm{z} \coloneqq \begin{cases}
                             d_X(e_X,z) & \text{if}\, z\in X,\\
                             d_Y(e_Y,z) & \text{if}\, z\in Y.
                            \end{cases}
        \end{align*}
  \item For $u,v\in X\sqcup Y$, set
        \begin{align*}
         d_{X\sqcup Y}(u,v) \coloneqq 
         \begin{cases}
          d_X(u,v) & \text{if}\, \{u,v\}\subset X,\\
          d_Y(u,v) & \text{if}\, \{u,v\}\subset Y,\\
          \infty & \text{else}.
         \end{cases}
        \end{align*}
 \end{itemize}
\end{notation}

We will construct the metric 
$d_*$ on the free product $X* Y$.
First, we define a symmetric function $D_*\colon W\times W\rightarrow \R$.
Let $\omega=uw_0\cdots w_n\in W$ and 
$\omega^{\prime}=uw_0^{\prime} \cdots w_m^{\prime}\in W$, 
where $u=u_0\dots u_k\in W$ is the maximal common prefix, that is,
we have $w_0 \neq w_0'$. 
We define 
\begin{align*}
 D_*(\omega, \omega^{\prime})\coloneqq 
 &d_{X\sqcup Y}(\overline{w_0},\overline{w_0^{\prime}}) 
 +\sum_{i=1}^n \norm{\overline{w_i}}
 +\sum_{j=1}^m \norm{\overline{w_i'}}
 +n + m +2.
\end{align*}
For $\omega=w_0\cdots w_n\in W$, we define
\begin{align*}
 D_*(\epsilon,\omega)\coloneqq  \sum_{i=1}^n \norm{\overline{w_i}} +n +1.
\end{align*}
Finally we define $D_*(\epsilon,\epsilon) = 0$.

Now we will define the metric $d_*$ on $X*Y$. First, we define $d_*$ on sheets.
Let $p,q\in X*Y$. We suppose that $p$ and $q$ belong to sheets.
Let $(\omega,u,0)$ and $(\tau,v,0)$ be the coordinate of $p$ and $q$, respectively, 
Here, $\omega,\tau \in W$ and $u,v\in X\sqcup Y$, as in Notation~\ref{nt:coord}.
We consider the following three cases.
\begin{enumerate}[label=(\Roman*)]
 \item \label{item:w=t}
       $\omega$ and $\tau$ are equal.

       In this case, first, we suppose that $\omega = \tau \neq \epsilon$.
       Then we define
       \begin{align*}
        d_*(p,q) \coloneqq d_{X\sqcup Y}(u,v).
       \end{align*}
       Next, we suppose that $\omega = \tau = \epsilon$.
       Then we define
       \begin{align*}
        d_*(p,q) \coloneqq \begin{cases}
                            d_{X\sqcup Y}(u,v) & \text{if} \, \{u,v\}\subset X 
                                               \text{ or } \{u,v\}\subset Y, \\
                            \norm{u} + \norm{v} + 1& \text{else} \quad \text{(see Figure \ref{fig 2}).}
                           \end{cases}
       \end{align*}
       \begin{figure}[h]
       \begin{tikzpicture}[scale=1]
       \draw(0,0)--++(4,0)--++(-1,-1)--++(-4,0)--cycle; 
       \draw(7,0)--++(4,0)--++(-1,-1)--++(-4,0)--cycle; 
       \draw(1,-0.5)to[out=45,in=135](8,-0.5);
       \draw(-1,-0.5)node[left]{$\epsilon X$};
       \draw(11,-0.5)node[right]{$\epsilon Y$};
       \draw(1.5,-0.5)node{$e_X$};
       \draw(0,-0.75)node[above]{$u$};
       \draw(7.5,-0.5)node{$e_Y$};
       \draw(9,-0.75)node[above]{$v$};
       \draw(4.5,1.5)node{$\{\epsilon \} \times [0,1]$};
       \fill(9,-0.75)circle(0.06);
       \fill(0,-0.75)circle(0.06);
       \end{tikzpicture}
              \caption{An example of case \ref{item:w=t}. 
                   Let $p=(\epsilon,u)$ and $q=(\epsilon, v)$, where $u \in X$, and $v \in Y$.}
       \label{fig 2}
       \end{figure}
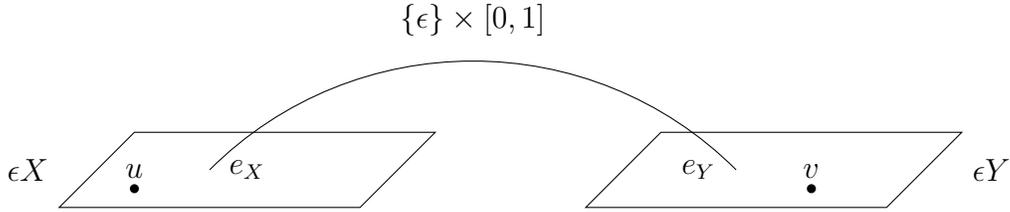
 \item \label{item:w-sub-t} 
$\omega$ is a proper subword of $\tau$.

       In this case, there exists $\tau'\in W\setminus \{\epsilon\}$ such that $\tau=\omega\tau'$. 
       Let $z$ be the initial letter of $\tau'$. First, we assume that $\omega$ is not the empty word.
       See Figure \ref{fig 3}. 
       In this figure, $z=x_0'$.
       Then we define 
       \begin{align*}
        d_*(p,q)\coloneqq d_{X\sqcup Y}(u,\overline{z})
        + D_*(\epsilon,\tau') + \norm{v}.
       \end{align*}

       Next, we assume that  $\omega$ is the empty word. 
       If $\{u,\overline{z}\}\subset X$ or $\{u,\overline{z}\}\subset Y$ holds, 
       then we define
       \begin{align*}
        d_*(p,q)\coloneqq d_{X\sqcup Y}(u,\overline{z})
        + D_*(\epsilon,\tau') + \norm{v}.
       \end{align*}
       Otherwise, we define
       \begin{align*}
        d_*(p,q)\coloneqq \norm{u} + \norm{\overline{z}} +D_*(\epsilon,\tau') + \norm{v} + 1.
       \end{align*}

 \item Neither \ref{item:w=t} nor \ref{item:w-sub-t}. \label{III}

       In this case, there exist the maximal common prefix $\rho\in W$ (possibly the empty word) and
       $\omega',\tau'\in W \setminus \{\epsilon\}$
       such that $\omega=\rho\omega'$ and $\tau = \rho\tau'$. 
       Let $z_0$ and $w_0$ be the initial letter of $\omega'$ and $\tau'$, respectively.
       See Figure \ref{fig 4}.
       In this figure, $z_0=x_0$ and $w_0=x_0'$. 
       Then we define
       \begin{align*}
        d_*(p,q) \coloneqq d_{*}((\rho,\overline{z_0}),(\rho,\overline{w_0})) + D_*(\epsilon, \omega^{\prime}) +D_*(\epsilon,\tau')+ \norm{u} + \norm{v}.
       \end{align*}
       
\end{enumerate}
       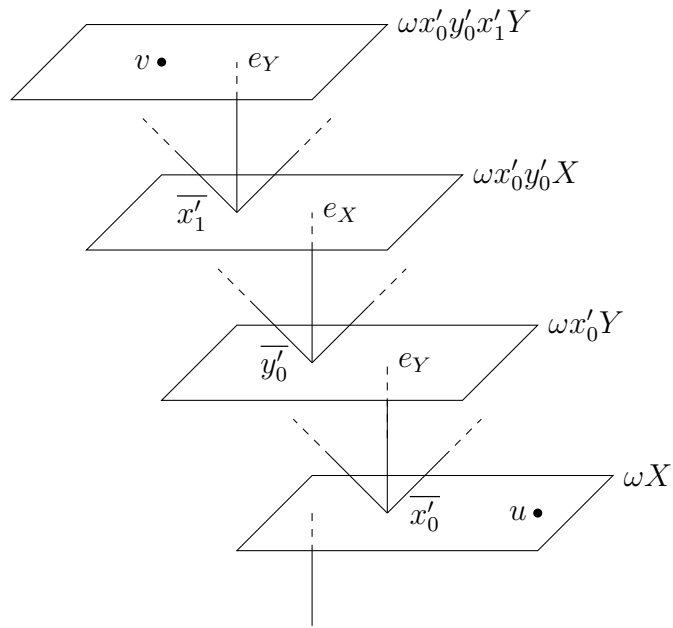
\begin{figure}[h]
       \vspace*{0.5cm}
       \begin{tikzpicture}
       \draw(-2,4)--++(4,0)--++(-1,-1)--++(-4,0)--cycle; 
       \draw(-1,3.5)node[left]{$v$};
       \fill(-1,3.5)circle(0.06);
       \draw(2,4)node[right]{$\omega x_0^{\prime} y_0^{\prime} x_1^{\prime} Y$};
       \draw(-1,2)--++(4,0)--++(-1,-1)--++(-4,0)--cycle; 
       \draw(3,2)node[right]{$\omega x_0^{\prime} y_0^{\prime} X$};
       \draw(0,0)--++(4,0)--++(-1,-1)--++(-4,0)--cycle; 
       \draw(4,0)node[right]{$\omega x_0^{\prime} Y$};
       \draw(1,-2)--++(4,0)--++(-1,-1)--++(-4,0)--cycle; 
       \draw(5,-2)node[right]{$\omega X$};
       \draw[dashed](1,-2.5)--++(0,-0.5);
       \draw(1,-3)--++(0,-1);
       \draw(2,-2.5)--++(0,1.5);
       \draw[dashed](2,-1.5)--++(0,1);
       \draw(2,-2.5)--++(0.75,0.75);
       \draw[dashed](2.75,-1.75)--++(0.5,0.5);
       \draw(2,-2.5)--++(-0.75,0.75);
       \draw[dashed](1.25,-1.75)--++(-0.5,0.5);
      \draw(2.5,-2.5)node{$\overline{x_0^{\prime}}$};
      \draw(2,-0.5)node[right]{$e_Y$};
      \draw(4,-2.5)node[left]{$u$};
      \fill(4,-2.5)circle(0.06);
       \draw(1,-0.5)--++(0,1.5);
       \draw[dashed](1,1)--++(0,0.5);
       \draw(1,-0.5)--++(0.75,0.75);
       \draw[dashed](1.75,0.25)--++(0.5,0.5);
       \draw(1,-0.5)--++(-0.75,0.75);
       \draw[dashed](0.25,0.25)--++(-0.5,0.5);
       \draw(0.5,-0.5)node{$\overline{y_0^{\prime}}$};
       \draw(1,1.5)node[right]{$e_X$};
       \draw(0,1.5)--++(0,1.5);
       \draw[dashed](0,3)--++(0,0.5);
       \draw(0,1.5)--++(0.75,0.75);
       \draw[dashed](0.75,2.25)--++(0.5,0.5);
       \draw(0,1.5)--++(-0.75,0.75);
       \draw[dashed](-0.75,2.25)--++(-0.5,0.5);
       \draw(0,3.5)node[right]{$e_Y$};
       \draw(-0.25,1.5)node[left]{$\overline{x_1^{\prime}}$};
       \end{tikzpicture}
              \caption{An example of case \ref{item:w-sub-t}. 
       Let $p=(\omega,u)$ and $q=(\omega  x_0^{\prime}y_0^{\prime}x_1^{\prime},v)$, where $\omega \in W_X \setminus \{\epsilon\}$, $x_i^{\prime} \in X_0^*$, $y_0^{\prime} \in Y_0^*$, $u \in X$, and $v \in Y$.}
       \label{fig 3}
       \end{figure}
       \begin{figure}[h]
       \begin{tikzpicture}[scale=1]
       \draw(0,0)--++(4,0)--++(-1,-1)--++(-4,0)--cycle; 
       \draw(8,0)--++(4,0)--++(-1,-1)--++(-4,0)--cycle; 
       \draw(4,0)node[right]{$\rho x_0 Y$};
       \draw(12,0)node[right]{$\rho x^{\prime}_0 Y$};
       \draw(4,-2)--++(4,0)--++(-1,-1)--++(-4,0)--cycle; 
       \draw(5,-2.5)--++(-1.75,1.75);
       \draw[dashed](3.25,-0.75)--++(-0.25,0.25);
       \draw(3,-0.5)node[left]{$e_Y$};
       \draw(6,-2.5)--++(1.5,1.5);
       \draw[dashed](7.5,-1)--++(0.25,0.25);
       \draw(7.75,-0.75)node[right]{$e_Y$};
       \draw[dashed](5.5,-2.5)--++(0,-0.5);
       \draw(5.5,-3)--++(0,-1);
       \draw(4.8,-2.5)node[left]{$\overline{x_0}$};
       \fill(5,-2.5)circle(0.06);
       \draw(6.5,-2.5)node{$\overline{x^{\prime}_0}$};
       \fill(6,-2.5)circle(0.06);
       \draw(8,-2)node[right]{$\rho X$};
       \draw(1,-0.5)--++(0,1.5);
       \draw[dashed](1,1)--++(0,0.5);
       \draw(1,-0.5)--++(0.75,0.75);%
       \draw[dashed](1.75,0.25)--++(0.5,0.5);
       \draw(1,-0.5)--++(-0.75,0.75);
       \draw[dashed](0.25,0.25)--++(-0.5,0.5);
       \draw(0.5,-0.5)node{$\overline{y_0}$};
       \draw(1,1.5)node[left]{$e_X$};
       \draw(2,1.5)node[above]{$u$};
       \fill(2,1.5)circle(0.06);
       \draw(10,-0.5)--++(0,1.5);
       \draw[dashed](10,1)--++(0,0.5);
       \draw(10,-0.5)--++(0.75,0.75);
       \draw[dashed](10.75,0.25)--++(0.5,0.5);
       \draw(10,-0.5)--++(-0.75,0.75);
       \draw[dashed](9.25,0.25)--++(-0.5,0.5);
       \draw(10.5,-0.5)node{$\overline{y^{\prime}_0}$};
       \draw(10,1.5)node[right]{$e_X$};
       \draw(8,2)--++(4,0)--++(-1,-1)--++(-4,0)--cycle; 
       \draw(12,2)node[right]{$\rho x^{\prime}_0 y^{\prime}_0 X$};
       \draw(8.5,1.5)--++(0,1.5);
       \draw[dashed](8.5,3)--++(0,0.5);
       \draw(8.5,1.5)--++(0.75,0.75);
       \draw[dashed](9.25,2.25)--++(0.5,0.5);
       \draw(8.5,1.5)--++(-0.75,0.75);
       \draw[dashed](7.75,2.25)--++(-0.5,0.5);
       \draw(8,1.5)node{$\overline{x^{\prime}_1}$};
       \draw(8.5,3.5)node[right]{$e_Y$};
       \draw(0,2)--++(4,0)--++(-1,-1)--++(-4,0)--cycle; 
       \draw(4,2)node[right]{$\rho x_0y_0 X$};
       \draw(8,4)--++(4,0)--++(-1,-1)--++(-4,0)--cycle; 
       \draw(12,4)node[right]{$\rho x^{\prime}_0 y^{\prime}_0 x^{\prime}_1 Y$};
       \draw(10,3.5)node[above]{$v$};
       \fill(10,3.5)circle(0.06);
       \end{tikzpicture}
              \caption{An example of \ref{III}.
       Let $p=(\rho x_0y_0, u)$ and $q=(\rho x^{\prime}_0y^{\prime}_0x^{\prime}_1, v)$,
       where $x_i,x_i^{\prime} \in X_0^*$, $y_0, y_0^{\prime} \in Y_0^*$, $u \in X$, and $v \in Y$.
       Here, $\rho \in W_X$ and $x_0 \neq x_0^{\prime}$.
       }
       \label{fig 4}
       \end{figure}
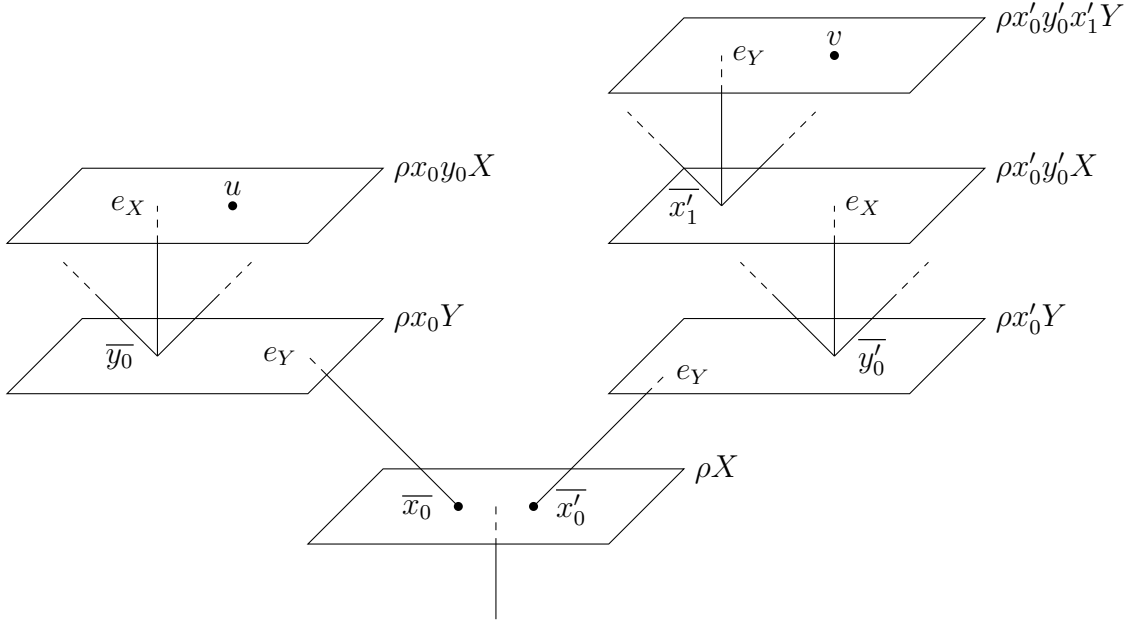
Finally, we extend $d_*$ on edges in an obvious way. 
From the construction, it is clear that $d_*$ is symmetric, 
non-degenerate, and satisfies the triangle inequality. 
This completes the construction of the metric $d_*$ on the free product $X*Y$.

\begin{example}
Let $G$ and $H$ be finitely generated groups
and let $S_G$ and $S_H$ be finite generating sets of $G$ and $H$, respectively.
Let $\Cay(G, S_G)$ and $\Cay(H, S_H)$ be Cayley graphs of $G$ and $H$.
We can construct the free product $\Cay(G, S_G)$ and $\Cay(H, S_H)$ with respect to $(G, \iota_G, e_G)$ and $(H, \iota_H, e_H)$, where $\iota_G$ and $\iota_H$ are the inclusion maps.
The resulting space $\Cay(G, S_G) * \Cay(H, S_H)$
coincides with the free product of Cayley graphs in the sense of Pisanski and Tucker \cite{PT}. 
\end{example}

The following example shows that the coarse geometry of a free product
depends on the choice of nets.
\begin{example}
\label{eg:finite_net}
For $m \in \mathbb{N}$, set $G_m=\mathbb{Z}/m \mathbb{Z}$. 
We assume that $G_m$ acts on $\mathbb{R}^2$ by rotations, that is,
\[
[l] \cdot (r\cos \theta, r \sin \theta)= \left(r \cos \left(\theta + \frac{2\pi l}{m} \right),  r \sin \left( \theta + \frac{2 \pi l}{m} \right) \right).
\]
Let $o((1,0))$ be an orbit map of $(1,0)$ by this action. 
The triplet $(G_m, o((1,0)), (1,0))$ is a net of $\mathbb{R}^2$. 
For $m,n \in \mathbb{N}$, we consider the free product $\mathbb{R}^2 * \mathbb{R}^2$ 
with respect to $(G_m, o((1,0)), (1,0))$ and $(G_n, o((1,0)), (1,0))$.
We will see how the growth type of $\mathbb{R}^2 * \mathbb{R}^2$ depends on
the choice of nets.

We denote by $V(R)$ the number of the points of the form $(\omega,(p,q))$ 
in the closed ball of radius $R$ centered at $(\epsilon,(0,0))$, where
$\omega$ is a word and $p,q$ are integers.

Set $G_m^*=G_m\setminus\{0\}$ and let $W$ be the set of normal words on 
$G_n^*\sqcup G_m^*$, and the empty word $\epsilon$.
\begin{enumerate}
\item 
Suppose $m=1$. In this case, $G_m^*$ is empty, so normal words consist only of
single letters, namely, $W=G_n^* \sqcup \{\epsilon\}$.
Then, the number of sheets is $n+1$. 
We illustrate the shape of $\mathbb{R}^2*\mathbb{R}^2$ in Figure \ref{spike}.
We can estimate $V(R)$ as follows:
\[
V(R) \leq (n+1) \cdot 2\pi R^2.
\]
Therefore, in this case, $\mathbb{R}^2*\mathbb{R}^2$ has a polynomial growth. 
       \begin{figure}[h]
       \begin{tikzpicture}[scale=1]
       \draw(-1.5,0)--++(4,0)--++(-1,-1)--++(-4,0)--cycle; 
       \draw(8,0)--++(4,0)--++(-1,-1)--++(-4,0)--cycle; 
       \draw(2,-2)--++(6,0)--++(-1,-1)--++(-6,0)--cycle; 
       \draw(2.9,-2.4)--++(-1.4,1.4);
       \draw[dashed](1.5,-1)--++(-0.25,0.25);
       \draw[dashed] (4.5,-2.5) circle [x radius=0.2,y radius=2,rotate=90];
       \draw(6.15,-2.35)--(7.5,-1);
       \draw[dashed](7.5,-1)--++(0.25,0.25);
       \draw(5.5,-2.3)--(7,2); 
       \draw[dashed](7,2)--(7.3,2.7);
       \draw(6,3)--++(4,0)--++(-1,-1)--++(-4,0)--cycle; 
       \draw(3.5,-2.3)--(1.8,2); 
       \draw[dashed](1.8,2)--(1.5,2.7);
       \draw(0,3)--++(4,0)--++(-1,-1)--++(-4,0)--cycle; 
       \fill(5,1)circle(0.06);
       \fill(4.5,1)circle(0.06);
       \fill(4,1)circle(0.06);
       \end{tikzpicture}
       \caption{$m=1.$}
       \label{spike}
       \end{figure}
\item
Suppose $m=2$ and $n=2$. 
In this case the set of letters $G_2^*\sqcup G_2^* = \{1\}\sqcup \{1\}$
contains two elements.
We illustrate the shape of $\mathbb{R}^2*\mathbb{R}^2$ in Figure \ref{gline}.
We can estimate $V(R)$:
\[
V(R) \leq R \cdot 2 \pi R^2. 
\]
Therefore, in this case, $\mathbb{R}^2*\mathbb{R}^2$ has a polynomial growth. 
 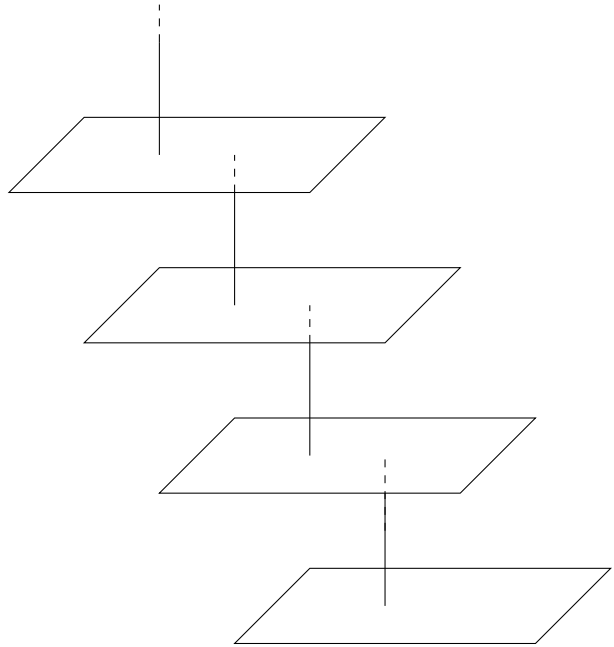
\begin{figure}[h]
       \vspace*{0.5cm}
       \begin{tikzpicture}
       \draw(-1,3.5)--++(0,1.5);
       \draw[dashed](-1,5)--++(0,0.5); 
       \draw(-2,4)--++(4,0)--++(-1,-1)--++(-4,0)--cycle; 
       \draw(-1,2)--++(4,0)--++(-1,-1)--++(-4,0)--cycle; 
       \draw(0,0)--++(4,0)--++(-1,-1)--++(-4,0)--cycle; 
       \draw(1,-2)--++(4,0)--++(-1,-1)--++(-4,0)--cycle; 
       \draw(2,-2.5)--++(0,1.5);
       \draw[dashed](2,-1.5)--++(0,1);
       \draw(1,-0.5)--++(0,1.5);
       \draw[dashed](1,1)--++(0,0.5);
       \draw(0,1.5)--++(0,1.5);
       \draw[dashed](0,3)--++(0,0.5);
       \end{tikzpicture}
       \caption{$m=2$ and $n=2$.}
       \label{gline}
       \end{figure}
\item
Suppose $m \geq 3$ and $n \geq 2$. In this case
$G_m^*$ has at least two elements. 
Then the free product has a tree-like structure. 
Since the growth of the free product of the groups
$G_m * G_n$ is exponential, the number of sheets grows exponentially.
Therefore, in this case, $\mathbb{R}^2*\mathbb{R}^2$ has a exponential growth. 
       \end{enumerate}

\end{example}

\begin{remark}\label{wp}
Let $X$, $X^{\prime}$, and $Y$ be metric spaces.
We suppose that $X$ and $X^{\prime}$ are quasi-isometric.
In general, $X*Y$ and $X^{\prime}*Y$ are not necessarily quasi-isometric.
For example, let $\Gamma_2$ be a Cayley graph of $\mathbb{Z}/2\mathbb{Z}$ for some generating set,
and $\Gamma_3$ be that of $\mathbb{Z}/3\mathbb{Z}$. It is clear that 
$\Gamma_2$ is quasi-isometric to $\Gamma_3$.
By Proposition~\ref{prop:GHactsonXY}, 
$\mathbb{Z}/2\mathbb{Z} * \mathbb{Z}/2\mathbb{Z}$ acts 
properly and cocompactly on $\Gamma_2 * \Gamma_2$, and so does
$\mathbb{Z}/3\mathbb{Z} * \mathbb{Z}/2\mathbb{Z}$ on $\Gamma_3*\Gamma_2$.
Since $\mathbb{Z}/2\mathbb{Z} * \mathbb{Z}/2\mathbb{Z}$ has two ends and
$\mathbb{Z}/3\mathbb{Z} * \mathbb{Z}/2\mathbb{Z}$ has infinitely many ends,
$\Gamma_2*\Gamma_2$ and $\Gamma_3*\Gamma_2$ are not quasi-isometric.
\end{remark}

\section{Coarsely convex spaces}
In this section, we briefly review coarse geometry and coarsely convex spaces.

\subsection{Geodesic metric spaces and proper metric spaces}
\par A metric space $(X,d_X)$ is a \textit{geodesic metric space}
if for any $x,x^{\prime} \in X$, there exists a map $\gamma : [0,a] \to X$ such that 
$\gamma(0)=x$, $\gamma(a)=x^{\prime}$, and $d_X(\gamma(t),\gamma(t^{\prime}))=|t-t^{\prime}|$ for all $t,t^{\prime} \in [0,a]$.
\par We say that a metric space $(X,d_X)$ is a \textit{proper metric space} if every bounded closed subset in $X$ is compact.
\subsection{Coarse equivalence and quasi-isometry}
Let $(X,d_X)$ and $(Y,d_Y)$ be metric spaces.
We say that a map $f : X \to Y$ is a \textit{coarse map} 
if there exist a non-decreasing function $\rho_+ : [0, \infty) \to [0, \infty)$ 
such that the inequality 
\begin{equation*}
 d_Y(f(x),f(x^{\prime})) \leq \rho_+(d_X(x,x^{\prime})) 
\end{equation*}
holds for any $x,x^{\prime} \in X$, and for any bounded subset $B\subset Y$, the 
preimage $f^{-1}(B)$ is bounded.

We also say that a map $f : X \to Y$ is a \textit{coarse embedding} 
if there exist non-decreasing functions $\rho_-, \rho_+ : [0, \infty) \to [0, \infty)$ 
such that 
\begin{equation*}
\lim_{t \to \infty} \rho_-(t)= \infty,
\end{equation*}
and the inequality 
\begin{equation*}
\rho_-(d_X(x,x^{\prime})) \leq d_Y(f(x),f(x^{\prime})) \leq \rho_+(d_X(x,x^{\prime})) 
\end{equation*}
holds for any $x,x^{\prime} \in X$.
When we can choose $\rho_-$ and $\rho_+$ to be affine maps, we say that the map $f$ is a \textit{quasi-isometric embedding}.
\par Let $X^{\prime} \subset X$. 
For $M\geq 0$, we say that $X^{\prime}$ is \textit{$M$-dense} in $X$ if $X=B_M(X^{\prime})$, where $B_M(X^{\prime})$ is the closed $M$-neighborhood of $X^{\prime}$.
We say that $X$ and $Y$ are \textit{coarsely equivalent} if there exist a coarse embedding map $f : X \to Y$ and $M\geq0$ 
such that $f(X)$ is $M$-dense in $Y$.
We say that $X$ and $Y$ are \textit{quasi-isometric} if there exist a quasi-isometric embedding map $f : X \to Y$ and $M\geq0$ 
such that $f(X)$ is $M$-dense in $Y$.
\par Let $\lambda \geq 1$ and $k\geq 0$.
A \textit{$(\lambda, k)$-quasi-geodesic segment} is a $(\lambda,k)$-quasi-isometric embedding $\gamma : [0,a] \to X$, 
that is, the inequality
\begin{equation*}
\lambda^{-1}| t-t^{\prime} | -k \leq d_X(\gamma(t),\gamma(t^{\prime})) \leq \lambda| t-t^{\prime} |+k
\end{equation*}
holds for all $t,t^{\prime} \in [0,a]$.

\subsection{Coarsely convex spaces}
\begin{definition}
Let $(X,d_X)$ be a metric space.
Let $\lambda \geq 1, k \geq 0, E \geq 1$, and $C \geq 0$ be constants.
Let $\theta : \mathbb{R}_{\geq 0} \to \mathbb{R}_{\geq 0}$ be a non-decreasing function.
Let $\mathcal{L}$ be a family of $(\lambda,k)$-quasi-geodesic segments.
The metric space $X$ is \textit{$(\lambda, k, E, C, \theta, \mathcal{L})$-coarsely convex}, if $\mathcal{L}$ satisfies the following:
\begin{enumerate}
\item[(CC1)] For any $v,w \in X$, there exists a quasi-geodesic segment $\gamma \in \mathcal{L}$ with $\gamma : [0,a] \to X$, $\gamma(0)=v$ and $\gamma(a)=w$.
\item[(CC2)] Let $\gamma, \eta \in \mathcal{L}$ be quasi-geodesic segments with $\gamma : [0,a] \to X$ and $\eta : [0,b] \to X$.
                 Then for all $t \in [0,a], s \in [0,b]$, and $c \in [0,1]$, we have that 
                \[ d_X(\gamma(ct), \eta(cs)) \leq (1-c)Ed_X(\gamma(0),\eta(0)) + cEd_X(\gamma(t), \eta(s))+C.\]
                We call this inequality \textit{the coarsely convex inequality}.
\item[(CC3)]Let $\gamma, \eta \in \mathcal{L}$ be quasi-geodesic segments with $\gamma : [0,a] \to X$ and $\eta : [0,b] \to X$.
                  Then for all $t \in [0,a]$ and $s \in [0,b]$, we have that 
                 \[ |t-s| \leq \theta( d_X(\gamma(0), \eta(0)) + d_X(\gamma(t), \eta(s))).\]
\end{enumerate}
The family $\mathcal{L}$ satisfying (CC1), (CC2), and (CC3) is called a \textit{system of good quasi-geodesic segments}, 
and elements $\gamma \in \mathcal{L}$ are called \textit{good quasi-geodesic segments}.
\end{definition}

We say that a metric space $X$ is \textit{a coarsely convex space}
if there exist constants $\lambda,k,E,C,$ a non-decreasing function  $\theta : \mathbb{R}_{\geq 0} \to \mathbb{R}_{\geq 0}$,
and a family of $(\lambda,k)$-quasi-geodesic segments $\mathcal{L}$ such that $X$ is $(\lambda, k, E, C, \theta, \mathcal{L})$-coarsely convex.
\par If the family $\mathcal{L}$ consists of only geodesic segments, then $\mathcal{L}$ always satisfies (CC3) by the triangle inequality.
\begin{lemma}\label{geodesic}
Let $(X,d_X)$ be a metric space.
Let $\mathcal{L}$ be a family of geodesic segments.
Then $\mathcal{L}$ satisfies $\mathrm{(CC3)}$.
In particular, we can take a non-decreasing function satisfying $\mathrm{(CC3)}$ to the identity map $\mathrm{id}_{\mathbb{R}_{\geq 0}}$. 
\end{lemma}
\begin{proof}
Let $(X,d_X)$ be a metric space
and let $\mathcal{L}$ be a family of geodesic segments.
Let $\gamma, \eta \in \mathcal{L}$ be geodesic segments with $\gamma : [0,a] \to X$ and $\eta : [0,b] \to X$.
Set $t \in [0,a]$ and $s \in [0,b]$.
We suppose that $t>s$.
Since $\gamma$ and $\eta$ are geodesic segments, we have
\begin{equation*}
|t-s|=d_X(\gamma(0),\gamma(t))-d_X(\eta(0),\eta(s)).
\end{equation*}
By the triangle inequality, we have
\begin{align*}
&d_X(\gamma(0),\gamma(t))-d_X(\eta(0),\eta(s)) \\
&\leq d_X(\gamma(0),\eta(0)) + d_X(\eta(0),\eta(s)) + d_X(\eta(s),\gamma(t)) - d_X(\eta(0),\eta(s))  \\
&=d_X(\gamma(0),\eta(0)) + d_X(\gamma(t),\eta(s)).
\end{align*}
Therefore the family $\mathcal{L}$ satisfies (CC3).
\end{proof}
Let $X$ be a metric space.
For a map $\gamma : [a,b] \to X$, we denote by $\gamma^{-1}$, the map $\gamma^{-1}:[a,b] \to X$ defined by $\gamma^{-1}(t)\coloneqq\gamma(b-(t-a))$ for $ t \in [a,b]$.
For $c \in [a,b]$, we denote by $\gamma|_{[a,c]}$ the restriction of $\gamma$ to $[a,c]$.
Let $\mathcal{L}$ be a family of quasi-geodesic segments in $X$.
The family $\mathcal{L}$ is \textit{symmetric} if $\gamma^{-1} \in \mathcal{L}$ for all $\gamma \in \mathcal{L}$,
and $\mathcal{L}$ is \textit{prefix-closed} if $\gamma|_{[a,c]} \in \mathcal{L}$ for all $\gamma \in \mathcal{L}$ with $\gamma: [a,b] \to X$ and for all $c \in [a,b]$.
\par
The following Proposition \ref{sgCC} plays an important role in the proof of the main result. 
\begin{proposition}\label{sgCC}
Let $(X,d_X)$ be a metric space.
Let $E \geq 1$ and $C \geq 0$ be constants.
Let $\mathcal{L}$ be a family of geodesic segments.
Suppose that $\mathcal{L}$ is symmetric and
$\mathcal{L}$ satisfies the following $\mathrm{(sgCC1)}$ and $\mathrm{(sgCC2)}$, then
$X$ is $(1, 0, E, 3C, \mathrm{id}_{\mathbb{R}_{\geq 0}}, \mathcal{L})$-coarsely convex.
\begin{itemize}
\item[$\mathrm{(sgCC1)}$] For any $v,w \in X$, there exists a geodesic segment $\gamma \in \mathcal{L}$ with $\gamma : [0,a] \to X$, $\gamma(0)=v$ and $\gamma(a)=w$.
\item[$\mathrm{(sgCC2)}$] Let $\gamma, \eta \in \mathcal{L}$ be geodesic segments with $\gamma : [0,a] \to X$ and $\eta : [0,b] \to X$.
Let $\gamma(0)=\eta(0)$.
Then for all $t \in [0,a], s \in [0,b]$, and $c \in [0,1]$, we have that 
\begin{equation*}
d_X(\gamma(ct), \eta(cs)) \leq cEd_X(\gamma(t), \eta(s))+C.
\end{equation*}
\end{itemize}
\end{proposition}

\begin{proof}
An assumption (sgCC1) implies that (CC1) holds. 
By Lemma \ref{geodesic}, a family of geodesic segments $\mathcal{L}$ satisfies (CC3). 
\par
We will show that $\mathcal{L}$ satisfies (CC2). 
Let $\gamma, \eta \in \mathcal{L}$ be geodesic segments with $\gamma:[0,a] \to X$ and $\eta:[0,b] \to X$.
Let $t \in [0,a], s \in [0,b]$, and $c \in [0,1]$.
By (sgCC1), there exist $\omega \in \mathcal{L}$ with  $\omega :[0,u] \to X$ such that $\omega(0)=\gamma(0)$ and $\omega(u)=\eta(s)$, 
and $\xi \in \mathcal{L}$ with $\xi :[0,v] \to X$ such that $\xi(0)=\eta(0)$ and $\xi(v)=\eta(s)$. 
Since the family of geodesic segments $\mathcal{L}$ is symmetric, $\omega^{-1}$ and $\xi^{-1}$ are in $\mathcal{L}$.
Note that $\gamma(0)=\omega(0)$ and $\omega^{-1}(0)=\xi^{-1}(0)$. 
Then, by \textrm{(sgCC2)}, we have 
\begin{align}
&d_X(\gamma(ct), \xi(cv)) \notag \\
&\leq d_X(\gamma(ct), \omega(cu)) + d_X(\omega(cu), \xi(cv)) \notag \\
&\leq cEd_X(\gamma(t),\omega(u))+C + d_X(\omega^{-1}((1-c)u), \xi^{-1}((1-c)v)) \notag \\
&\leq cEd_X(\gamma(t),\omega(u))+C + (1-c)Ed_X(\omega^{-1}(u),\xi^{-1}(v)) + C \notag \\
&= cEd_X(\gamma(t),\eta(s))+ (1-c)Ed_X(\gamma(0),\eta(0)) + 2C, \label{kirino1}
\end{align}
and we have
\begin{align}
&d_X(\xi(cv),\eta(cs)) \notag \\
&\leq cEd_X(\xi(v),\eta(s)) + C \notag \\
&=cEd_X(\eta(s),\eta(s)) + C \notag \\
&=C. \label{kirino2}
\end{align}
Combinig \eqref{kirino1} and \eqref{kirino2} yields
\begin{align*}
d_X(\gamma(ct),\eta(cs)) 
&\leq d_X(\gamma(ct),\xi(cv)) + d_X(\xi(cv),\eta(cs)) \\
&\leq (1-c)Ed_X(\gamma(0),\eta(0)) +cEd_X(\gamma(t),\eta(s)) +  3C.
\end{align*}
Therefore, $\mathcal{L}$ satisfies (CC2). 
This completes the proof.
\end{proof}

We say that $X$ is a \textit{symmetric geodesic coarsely convex space} 
if there exist constants $E$, $C$, and a family of geodesic segments $\mathcal{L}$ such that $\mathcal{L}$ is symmetric and satisfies (sgCC1) and (sgCC2).

\section{Main results}
In this section, we consider the free products of symmetric geodesic coarsely convex spaces.

\begin{theorem}\label{Ma}
Let $X$ and $Y$ be metric spaces with nets.
If $X$ and $Y$ are symmetric geodesic coarsely convex spaces,
then the free product $X*Y$ is a symmetric geodesic coarsely convex space.
\end{theorem}
In the rest of this section, we give the proof of Theorem~\ref{Ma}.
First, we construct the family of geodesic segments $\mathcal{L}_{*}$ in $X*Y$.
Let $(X_0,i_X,e_X)$ and $(Y_0,i_Y,e_Y)$ be nets of $X$ and $Y$, respectively, 
and let $\mathcal{S}$ be the sheets of $X*Y$.
Since $X$ and $Y$ are symmetric geodesic coarsely convex,
we can suppose that $X$ and $Y$ are $(1,0,E_X,C_X,\textrm{id}_{\mathbb{R}_{\geq 0}}, \mathcal{L}_X)$-coarsely convex 
and $(1,0,E_Y,C_Y,\textrm{id}_{\mathbb{R}_{\geq 0}}, \mathcal{L}_Y)$-coarsely convex, respectively.
We define
\begin{gather*}
E\coloneqq\max \{ E_X,E_Y \}, \\
C\coloneqq\max \{ C_X,C_Y \}.
\end{gather*}
\par 
For any $S \in \mathcal{S}$ and for any $x, y \in S$, we choose a good geodesic in $S$ from $x$ to $y$, denoted by $\gamma_{S}(x,y)$.
Let $E$ be the set of edges.
For $e \in E$ and $v,w \in e$, 
let $\gamma_{e}(v,w)$ be the geodesic segment on $e$ from $v$ to $w$. 
\par
Since $X$ and $Y$ are symmetric, we can choose it to be $\gamma_{S}(y,x)= \gamma_{S}(x,y)^{-1}$.
Note that $e \in E$ is a unit interval.
Then, for $e \in E$ and $v,w \in e$, we have $\gamma_{e}(w,v)=\gamma_{e}(v,w)^{-1}$. 
We define 
\[ \mathcal{L}\coloneqq\{\gamma_{S}(x,y) \mid S \in \mathcal{S}, \ x,y \in S\}
\bigsqcup
\{
\gamma_{e}(v,w) \mid e \in E, v,w \in e
\} \]
\par 
For $a,b \in X*Y$, we define a geodesic segment from $a$ to $b$
by connecting $\gamma \in \mathcal{L}$ at the identified points, denoted by $\Gamma(a,b)$.
Firstly, we suppose that $a$ and $b$ belong to sheets $S_1$ and $S_2$, respectively.
Let $e_{S_1}$ and $e_{S_2}$ be base points of $S_1$ and $S_2$ 
and let $(\omega, u)$ and $(\tau, v)$ be coordinate of $a$ and $b$, respectively.
Here, $\omega,\tau \in W$ and $u,v\in X\sqcup Y$, as in Notation~\ref{nt:coord}.
We consider the following three cases.
\begin{enumerate}[label=(\Roman*)]
 \item \label{item:w=t1}
       $\omega$ and $\tau$ are equal.

       In this case, first, we suppose that $\omega = \tau \neq \epsilon$ or $S_1 =S_2$. Then we define
       \begin{align*}
       \Gamma(a,b)(t) \coloneqq \gamma_{S_1}(u,v)(t). 
       \end{align*}
       Next, we suppose that $\omega = \tau = \epsilon$ and $S_1 \neq S_2$. 
       Let $e_{\epsilon}=\{\epsilon\} \times [0,1]$.
       Then we define
       \begin{align*}
       \Gamma(a,b)(t) \coloneqq \begin{cases}
                            \gamma_{S_1}(u,e_{S_1})(t) &  \text{for}\  0 \leq t \leq \norm{u}, \\
                            \gamma_{e_{\epsilon}}(e_{S_1}, e_{S_2})(t-t_1)                &  \text{for}\  t_1 \leq t \leq t_1+1, \\
                            \gamma_{S_2}(e_{S_2},v)(t-t_2)     & \text{for}\ t_2 \leq t \leq t_2 +\norm{v},
                           \end{cases}
       \end{align*}
       where $t_1=\norm{u}$ and $t_2=t_1+1$.
 \item \label{item:w-sub-t1} 
$\omega$ is a proper subword of $\tau$.

       In this case, there exists $\tau'\in W\setminus \{\epsilon\}$ such that $\tau=\omega\tau'$. 
       Let $z$ be the initial letter of $\tau'$. 
       We define the geodesic segment from $a$ to $b$ inductively for the length of $\tau^{\prime}$.
       First, we assume that $\omega$ is not the empty word and the height of $\tau^{\prime}$ is equal to 1, that is, $\tau^{\prime}=z$.
       Let $e=\{\omega\} \times \{z\} \times [0,1]$
       Then we define 
       \begin{align*}
        \Gamma(a,b)(t)\coloneqq 
        \begin{cases}
       \gamma_{S_1}(u, \overline{z})(t) & \text{for}\ 0 \leq t \leq d_*(u,\overline{z}),\\ 
       \gamma_{e}(\overline{z},e_{S_2})(t-t_1)                   & \text{for}\ t_1 \leq t \leq t_1+1,\\
       \gamma_{S_2}(e_{S_2}, v)(t-t_2) &\text{for}\ t_2 \leq t \leq t_2+ \norm{v},
        \end{cases}
       \end{align*}
       where $t_1=d_*(u,\overline{z})$ and $t_2=t_1+1$.
       Next, we assume that  $\omega$ is the empty word and $\tau^{\prime}=z$.
       Let $e'=\{\epsilon\} \times \{z\} \times [0,1]$.
       If $\{u,\overline{z}\}\subset X$ or $\{u,\overline{z}\}\subset Y$ holds, 
       then we define
       \begin{align*}
      \Gamma(a,b)(t) \coloneqq 
      \begin{cases}
       \gamma_{S_1}(u, \overline{z})(t) & \text{for}\ 0 \leq t \leq d_*(u,\overline{z}),\\ 
       \gamma_{e'}(\overline{z}, e_{S_2})(t-t_1)                   & \text{for}\ t_1 \leq t \leq t_1+1,\\
       \gamma_{S_2}(e_{S_2}, v)(t-t_2) &\text{for}\ t_2 \leq t \leq t_2+ \norm{v},
        \end{cases}
       \end{align*}
       where $t_1=d_*(u,\overline{z})$ and $t_2=t_1+1$.
       Otherwise, then we define
       \begin{align*}
        \Gamma(a,b)(t) \coloneqq \begin{cases}
       \gamma_{S_1}(u, e_{S_1})(t) & \text{for}\ 0 \leq t \leq \norm{u},\\ 
       \gamma_{e_{\epsilon}}(e_{S_1},e_{S_0})(t-t_1)                & \text{for}\ t_1 \leq t \leq t_1+1,\\
       \gamma_{S_0}(e_{S_0}, \overline{z})(t-t_2) &\text{for}\ t_2 \leq t \leq t_2+ \norm{\overline{z}},\\
       \gamma_{e'}(\overline{z}, e_{S_2})(t-t_3) &\text{for}\ t_3 \leq t \leq t_3+1, \\
       \gamma_{S_2}(e_{S_2},v)(t-t_4) &\text{for}\ t_4 \leq t \leq t_4+ \norm{v},
        \end{cases}
       \end{align*}
       where $t_1=\norm{u}$, $t_2=t_1+1$, $t_3=t_2+ \norm{\overline{z}}$, and $t_4=t_3+1$.
       Here, $S_0=\epsilon X$ or $S_0=\epsilon Y$. 
       We denote by $e_{S_0}$, the base point of $S_0$.
       From the construction, it is clear that $\Gamma(a,b)$ is a geodesic segment.
       Moreover, since $\gamma_S(y,x)=\gamma_S(x,y)^{-1}$ for any $S \in \mathcal{S}$ 
       and $\gamma_{e}(w,v)=\gamma_{e}(v,w)^{-1}$ for any $e \in E$, 
       then we have $\Gamma(b,a)=\Gamma(a,b)^{-1}$. 
       \par Finally, we suppose that the length of $\tau^{\prime}$ is greater than $1$.
       Let $\tau^{\prime}=l_0l_1l_2 \dots l_n$
       and let $a^{\prime}=( \omega l_0l_1 \dots l_{n-1}, \overline{l_n})$.
       Let $\hat{e}=\{\omega l_0l_1 \dots l_{n-1}\} \times \{l_n\} \times [0,1]$. 
       By the assumption of induction, $\Gamma(a, a^{\prime})$ is defined.
       Then we define $\Gamma(a,b)$ as follows:
       \begin{align*}
       \Gamma(a,b) \coloneqq \begin{cases}
       \Gamma(a,a^{\prime})(t) & \text{for}\ 0 \leq t \leq d_*(a,a^{\prime}), \\
       \gamma_{\hat{e}}(\overline{l_n}, e_{S_2})(t-t_1) &\text{for}\ t_1 \leq t \leq t_1+1, \\
       \gamma_{S_2}(e_{S_2}, v)(t-t_2) &\text{for}\ t_2 \leq t \leq t_2 + \norm{v}, \\
       \end{cases}
       \end{align*}
       where $t_1=d_*(a,a^{\prime})$, $t_2=t_1+1$.
 \item Neither \ref{item:w=t1} nor \ref{item:w-sub-t1}.

       In this case, there exist the maximal common prefix $\rho\in W$ (possibly the empty word) and
       $\omega',\tau'\in W \setminus \{\epsilon\}$
       such that $\omega=\rho\omega'$ and $\tau = \rho\tau'$. 
       Let $z_0$ and $w_0$ be the initial letter of $\omega'$ and $\tau'$, respectively.
       Set $a^{\prime}=(\rho, \overline{z_0})$ and $b^{\prime}=(\rho, \overline{w_0})$.
       Then we define 
       \begin{align*}
        \Gamma(a,b) \coloneqq \begin{cases} 
        \Gamma(a, a^{\prime})(t) & \text{for}\ 0 \leq t \leq d_*(a,a^{\prime}), \\
        \Gamma(a^{\prime}, b^{\prime})(t-t_1)& \text{for}\ t_1 \leq t \leq t_1+d_*(\overline{z_0}, \overline{w_0}), \\
        \Gamma(b^{\prime},b)(t-t_2) & \text{for}\ t_2 \leq t \leq t_2+d_*(b^{\prime},b),
        \end{cases}
       \end{align*}
       where $t_1=d_*(a,a^{\prime})$ and $t_2=t_1+d_*(\overline{z_0}, \overline{w_0})$.
      
\end{enumerate}
 Finally, we extend $\Gamma(a,b)$ on edges in an obvious way. 
 We denote by $[a,b]$, the image of the geodesic segment $\Gamma(a,b)$.
Define
\[ \mathcal{L}_{*}\coloneqq\{ \Gamma(a,b) \mid a,b \in X*Y\}. \] 
By the definition of $\Gamma(a,b) \in \mathcal{L}_{*}$, we have $\Gamma(b,a)=\Gamma(a,b)^{-1}$.
Therefore, the family of geodesic segments $\mathcal{L}_{*}$ is symmetric.

\begin{example}
Let $a=(ux_0y_0,x)$ and $b=(ux_0^{\prime},y)$, 
where $u \in W_X$ is the maximal common prefix, $x_0,  x^{\prime}_0 \in X_0^*$, $y_0 \in Y_0^*$, $x \in X$, and $y \in Y$.
Let $e_1=\{ux_0\} \times \{y_0\} \times [0,1]$,
$e_2=\{u\} \times \{x_0\} \times [0,1]$, and
$e_3=\{u\} \times \{x_0'\} \times [0,1]$. 
Then the path $\Gamma(a,b) : [0, d_*(a,b)] \to X*Y$ 
is given as follows (see Figure \ref{gab}):
\begin{equation*}
\Gamma(a,b)(t) = \begin{cases} 
              \gamma_{ux_0y_0X}(x, e_X)(t) & 0 \leq t \leq d_X(e_X,x),\\
              \gamma_{e_1}(e_X, \overline{y_0})(t-t_1) & t_1 \leq t \leq t_1 +1, \\ 
              \gamma_{ux_0Y}(\overline{y_0}, e_Y)(t- (t_1+1)) & t_1 +1 \leq t \leq t_1+d_Y(e_Y,\overline{y_0})+1, \\
              \gamma_{e_2}(e_Y, \overline{x_0})(t-t_2) & t_2  \leq t \leq t_2 +1, \\
              \gamma_{uX}(\overline{x_0}, \overline{x_0^{\prime}})(t-(t_2+1)) & t_2+1 \leq t \leq t_2 + d_X(\overline{x_0},\overline{x_0^{\prime}})+1, \\
              \gamma_{e_3}(\overline{x_0'},e_Y)(t-t_3) & t_3 \leq t\leq t_3+1, \\
              \gamma_{ux_0^{\prime}Y}(e_Y,y)(t-(t_3+1)) & t_3 +1 \leq t \leq t_3+d_Y(e_Y,y)+1,
                   \end{cases}        
\end{equation*}
where $t_1=d_X(e_X,x)$, $t_2=t_1+d_Y(e_Y,\overline{y_0})+1$, and $t_3=t_2 + d_X(\overline{x_0},\overline{x_0^{\prime}})+1$. 
       \begin{figure}[h]
       \begin{tikzpicture}[scale=1]
       \draw(0,0)--++(4,0)--++(-1,-1)--++(-4,0)--cycle; 
       \draw(8,0)--++(4,0)--++(-1,-1)--++(-4,0)--cycle; 
       \draw(4,0)node[right]{$u x_0 Y$};
       \draw(12,0)node[right]{$u x^{\prime}_0 Y$};
       \draw(4,-2)--++(4,0)--++(-1,-1)--++(-4,0)--cycle; 
       \draw(5,-2.5)--++(-1.75,1.75);
       \draw[dashed](3.25,-0.75)--++(-0.25,0.25);
       \draw(6,-2.5)--++(1.5,1.5);
       \draw[dashed](7.5,-1)--++(0.25,0.25);
       \draw(8,-0.75)node[above]{$e_Y$};
       \fill(7.75,-0.75)circle(0.06);
       \draw(9,-0.75)node[above]{$y$};
       \fill(9,-0.75)circle(0.06);
       \draw(7.75,-0.75)--(9,-0.75);
       \draw[dashed](4,-2.5)--++(0,-0.5);
       \draw(4,-3)--++(0,-1);
       \draw(4.9,-2.5)node[left]{$\overline{x_0}$};
       \fill(5,-2.5)circle(0.06);
       \draw(6.5,-2.5)node{$\overline{x^{\prime}_0}$};
       \fill(6,-2.5)circle(0.06);
       \draw(5,-2.5)--(6,-2.5);
       \draw(8,-2)node[right]{$u X$};
       \draw(1,-0.5)--++(0,1.5);
       \draw[dashed](1,1)--++(0,0.5);
       \draw(1,-0.5)--++(0.75,0.75);%
       \draw[dashed](1.75,0.25)--++(0.5,0.5);
       \draw(1,-0.5)--++(-0.75,0.75);
       \draw[dashed](0.25,0.25)--++(-0.5,0.5);
       \draw(3,-0.5)node[above]{$e_Y$};
       \draw(0.5,-0.5)node{$\overline{y_0}$};
       \draw(3,-0.5)--(1,-0.5);
       \fill(3,-0.5)circle(0.06);
       \fill(1,-0.5)circle(0.06);
       \draw(1,1.5)node[left]{$e_X$};
       \draw(2,1.5)node[above]{$x$};
       \draw(2,1.5)--(1,1.5);
       \fill(1,1.5)circle(0.06);
       \fill(2,1.5)circle(0.06);
       \draw(0,2)--++(4,0)--++(-1,-1)--++(-4,0)--cycle; 
       \draw(4,2)node[right]{$u x_0y_0 X$};
       \end{tikzpicture}
       \caption{An example of $\Gamma(a,b) \in \mathcal{L}_*$}
       \label{gab}
       \end{figure}
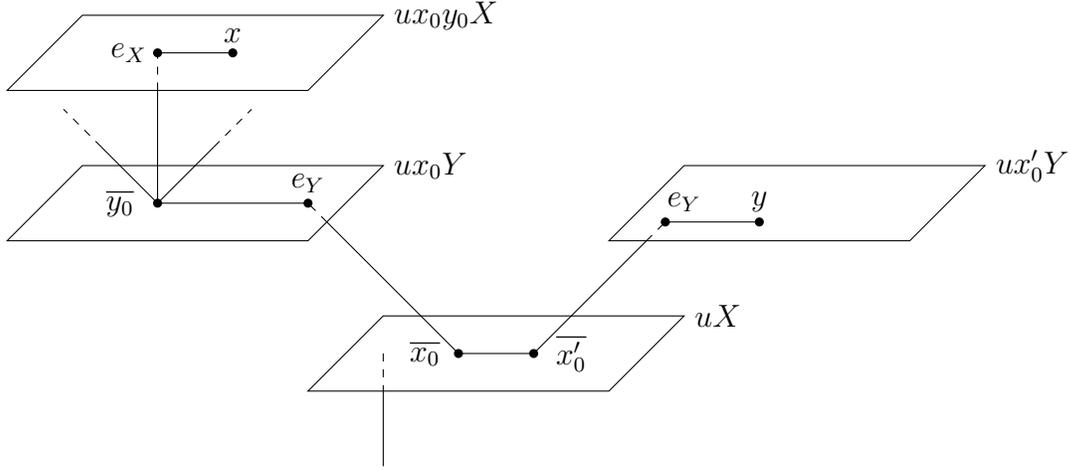
\end{example}

\begin{proposition}\label{Propo L}
The family of geodesic segments $\mathcal{L}_*$ satisfies $\mathrm{(sgCC1)}$ and $\mathrm{(sgCC2)}$.
\end{proposition}
\begin{proof}
We remark that $X$ and $Y$ are geodesic spaces.
Then, by the definition of $\mathcal{L}_{*}$, it is clear that $\mathcal{L}_{*}$ satisfies (sgCC1).
\par We will prove that $\mathcal{L}_{*}$ satisfies  (sgCC2).
Let $\Gamma_1 \in \mathcal{L}_{*}$ with $\Gamma_1: [0,t] \to X*Y$ and let $\Gamma_2 \in \mathcal{L}_{*}$ with $\Gamma_2:[0,s] \to X*Y$.
We suppose that $\Gamma_1(0)=\Gamma_2(0)$.
The geodesic triangle $\Delta(\Gamma_1(0), \Gamma_1(t), \Gamma_2(s))$ has the following form (see Figure \ref{gt}):
There exist $S \in \mathcal{S}$ and $p, a, b \in S$ such that
\begin{gather*}
[\Gamma_1(0), \Gamma_1(t)]= [\Gamma_1(0), p] \cup [p,a] \cup [a, \Gamma_1(t)], \\ 
[\Gamma_1(0), \Gamma_2(s)]= [\Gamma_1(0), p] \cup [p,b] \cup [b, \Gamma_2(s)], \\
[\Gamma_1(t), \Gamma_2(s)]= [\Gamma_1(t), a] \cup [a,b] \cup [b, \Gamma_2(s)]. 
\end{gather*}
We set 
\begin{align*}
p&\coloneqq\Gamma_1(c_0t)=\Gamma_2(c_0^{\prime}s), \\
a&\coloneqq\Gamma_1(c_1t), \\
b&\coloneqq\Gamma_2(c_1^{\prime}s),
\end{align*}
where $c_0, c_1, c_0^{\prime}, c_1^{\prime} \in [0,1]$.
\begin{figure}[h]
\begin{tikzpicture}
\draw(0,0)--++(6,0)--++(-1,-1.5)--++(-6,0)--cycle; 
\draw(5,-1.5)node[below]{$S$};
\draw[dashed](1,-1)--++(-0.5,-0.5);
\draw(0.5,-1.5)--++(-1.5,-1.5);
\draw(-1,-3)node[below]{$\Gamma_1(0)$};
\fill(-1,-3)circle(0.06);
\draw(1,-1)--++(2,0.75)--++(0.75,-0.75)--cycle;
\draw(1,-1)node[below]{$p$};
\draw(3.5,-0.25)node[left]{$a$};
\draw(3.75,-1)node[below]{$b$};
\draw(3,-0.25)--++(1,2);
\draw(4,1.75)node[above]{$\Gamma_1(t)$};
\fill(4,1.75)circle(0.06);
\draw(3.75,-1)--++(3,1);
\draw(6.75,0)node[above]{$\Gamma_2(s)$};
\fill(6.75,0)circle(0.06);
\end{tikzpicture}
\caption{An example of geodesic triangles of $X*Y$.}
\label{gt}
\end{figure}
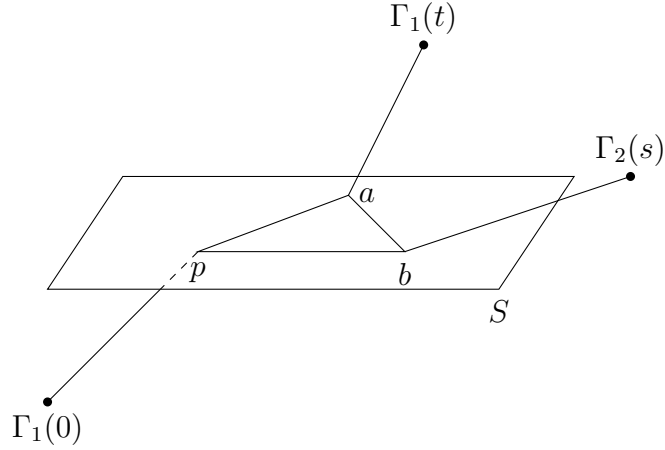
Since $\Gamma_1$ and $\Gamma_2$ are geodesics, we have $c_0t=c_0^{\prime}s$.
Without loss of generality, we may assume that $c_1 \leq c_1^{\prime}$.
Since $\Gamma_1, \Gamma_2$ are geodesics with respect to the metric $d_*$,
we have 
\begin{align*}
 & c_1^{\prime} \cdot d_*(\Gamma_1(t), \Gamma_2(s)) - d_*(a,b) \\
 &= c_1^{\prime}\{ d_*(\Gamma_1(t), a) + d_*(a,b) + d_*(b, \Gamma_2(s))\} - d_*(a,b)\\
                                                                         & =  c_1^{\prime}d_*(\Gamma_1(t), a) + c_1^{\prime}d_*(b, \Gamma_2(s)) - (1-c_1^{\prime})d_*(a,b) \\
                                                                         & =  c_1^{\prime}d_*(\Gamma_1(t), \Gamma_1(c_1t)) + c_1^{\prime}d_*(\Gamma_2(c_1^{\prime}s), \Gamma_2(s)) - (1-c_1^{\prime})d_*(a,b) \\
                                                                         & =  c_1^{\prime}(1-c_1)d_*(\Gamma_1(t), \Gamma_1(0)) + c_1^{\prime}(1-c_1^{\prime})d_*(\Gamma_2(0), \Gamma_2(s)) - (1-c_1^{\prime})d_*(a,b). 
\end{align*}
Moreover, by $c_1 \leq c_1^{\prime}$, the right-hand side of this equality can be estimated as follows:
\begin{align*}
& c_1^{\prime}(1-c_1)d_*(\Gamma_1(t), \Gamma_1(0)) + c_1^{\prime}(1-c_1^{\prime})d_*(\Gamma_2(0), \Gamma_2(s)) - (1-c_1^{\prime})d_*(a,b) \\
                                                                         &\geq c_1^{\prime}(1-c_1^{\prime})d_*(\Gamma_1(t), \Gamma_1(0)) + c_1^{\prime}(1-c_1^{\prime})d_*(\Gamma_2(0), \Gamma_2(s)) - (1-c_1^{\prime})d_*(a,b) \\
                                                                         &= (1-c_1^{\prime})\{ c_1^{\prime}d_*(\Gamma_1(t), \Gamma_1(0)) + c_1^{\prime}d_*(\Gamma_2(0), \Gamma_2(s)) - d_*(a,b)\} \\
                                                                         &\geq (1-c_1^{\prime})\{ c_1d_*(\Gamma_1(t), \Gamma_1(0)) + c_1^{\prime}d_*(\Gamma_2(0), \Gamma_2(s)) - d_*(a,b)\} \\
                                                                         &= (1-c_1^{\prime})\{ d_*(a, \Gamma_1(0)) + d_*(\Gamma_2(0),b) - d_*(a,b)\} \\
                                                                         &\geq 0.
\end{align*}
The last inequality follows from the triangle inequality.
Then, we have 
\begin{equation}\label{key}
 d_*(a,b) \leq c_1^{\prime} \cdot d_*(\Gamma_1(t), \Gamma_2(s)).
\end{equation}
\par We will show that 
there exist $E_{*} \geq 1$ and $C_* \geq 0$ depending only on $E$ and $C$ such that
for all $c \in [0,1]$, 
\[ d_*(\Gamma_1(ct), \Gamma_2(cs)) \leq cE_*d_*(\Gamma_1(t), \Gamma_2(s))+C_*.\]
We divide the proof into the following cases: 
\begin{enumerate}
\renewcommand{\labelenumi}{\Roman{enumi}).}
\item $\Gamma_1(ct) \in [\Gamma_2(0),\Gamma_2(cs)]$, or $\Gamma_2(cs) \in [\Gamma_1(0),\Gamma_1(ct)]$.
\item \label{itemI)} $\Gamma_1(ct) \in [p,a]$ and $\Gamma_2(cs) \in [p,b]$ (see Figure \ref{I)}).
\item \label{itemII)}$\Gamma_1(ct) \in [p,a]$ and $\Gamma_2(cs) \in [b, \Gamma_2(s)]$ (see Figure \ref{II)}).
\item \label{itemIII)}$\Gamma_1(ct) \in [a, \Gamma_1(t)]$ and $\Gamma_2(cs) \in [b, \Gamma_2(s)]$ (see Figure \ref{III)}).
\end{enumerate}
\par 
We consider case I).
Note that $\Gamma_1(0)=\Gamma_2(0)$.
In case I), we suppose that $\Gamma_1(0)$, $\Gamma_1(ct)$, and $\Gamma_2(cs)$ are on the same geodesic segment.
Then we have that
\begin{align*}
d_*(\Gamma_1(ct),\Gamma_2(cs))&=c|t-s|\\
                                              &\leq c d_*(\Gamma_1(t),\Gamma_2(s)).
\end{align*}
\par
In case II), we can put
\begin{align*}
\Gamma_1(ct)&=\gamma_{S}(p,a)(ct-c_0t), \\
\Gamma_2(cs)& =\gamma_{S}(p,b)(cs-c_0^{\prime}s).
\end{align*}
\begin{figure}[h]
\begin{tikzpicture}
\draw(0,0)--++(2,1)--++(0,-2)--cycle;
\draw(-2,0)--++(2,0);
\draw(-2,0)node[left]{$\Gamma_1(0)$};
\draw(0,0)node[above]{$p$};
\fill(-2,0)circle(0.06);
\draw(2,1)--++(1,1.5);
\draw(2,1)node[above]{$a$};
\draw(3,2.5)node[above]{$\Gamma_1(t)$};
\fill(3,2.5)circle(0.06);
\draw(2,-1)--++(1,-1.5);
\draw(2,-1)node[right]{$b$};
\draw(3,-2.5)node[below]{$\Gamma_2(s)$};
\fill(3,-2.5)circle(0.06);
\draw(0.8,0.5)node[above]{$\Gamma_1(ct)$};
\fill(1,0.5)circle(0.06);
\draw(1,-1)node{$\Gamma_2(cs)$};
\fill(1,-0.5)circle(0.06);
\end{tikzpicture}
\caption{Case II)}
\label{I)}
\end{figure}
Set $T\coloneqq(c_1-c_0)t$ and $S\coloneqq(c_1^{\prime}-c_0^{\prime})s$.
Then we have
\begin{align*}
p&=\gamma_{S}(p,a)\left( 0 \right)=\gamma_{S}(p,b)\left( 0 \right), \\
\Gamma_1(ct)&=\gamma_{S}(p,a)\left( \dfrac{c-c_0}{c_1-c_0}T\right),                                     & a &=\gamma_{S}(p,a)\left(T\right), \\ 
\Gamma_2(cs)&=\gamma_{S}(p,b)\left( \dfrac{c-c_0^{\prime}}{c_1^{\prime}-c_0^{\prime}}S\right), & b &=\gamma_{S}(p,b)\left(S\right).
\end{align*}
We define $a^{\prime} \in [p,a]$ to be
\[ a^{\prime}\coloneqq \gamma_{S}(p,a)\left( \dfrac{c-c_0^{\prime}}{c_1^{\prime}-c_0^{\prime}}T\right) .\]
Firstly, by the coarsely convex inequality in $S$ and inequality \eqref{key}, we have
\begin{align}
d_*(a^{\prime}, \Gamma_2(cs)) 
&=d_* \left( \gamma_{S}(p,a)\left( \dfrac{c-c_0^{\prime}}{c_1^{\prime}-c_0^{\prime}}T\right) , \gamma_{S}(p,b)\left( \dfrac{c-c_0^{\prime}}{c_1^{\prime}-c_0^{\prime}}S\right) \right)  \notag \\
&\leq \dfrac{c-c_0^{\prime}}{c_1^{\prime}-c_0^{\prime}} \cdot Ed_*(\gamma_{S}(p,a)(T), \gamma_{S}(p,b)(S)) +C\notag \\
&= \dfrac{c-c_0^{\prime}}{c_1^{\prime}-c_0^{\prime}} \cdot Ed_*(a, b) +C\notag \\
&\leq \dfrac{cc_1^{\prime}-c_0^{\prime}c_1^{\prime}}{c_1^{\prime}-c_0^{\prime}} \cdot Ed_*(\Gamma_1(t), \Gamma_2(s)) +C\notag \\
&\leq c \cdot Ed_*(\Gamma_1(t), \Gamma_2(s)) +C. \label{1}
\end{align}
Inequality \eqref{1} follows from $c_0^{\prime} \leq c \leq c_1^{\prime}$.
\par Next, we consider the distance between $a^{\prime}$ and $\Gamma_1(ct)$.
Then
\begin{align*}
d_*(a^{\prime}, \Gamma_1(ct)) 
&=d_* \left( \gamma_{S}(p,a)\left( \dfrac{c-c_0^{\prime}}{c_1^{\prime}-c_0^{\prime}}T\right) , \gamma_{S}(p,a)\left( \dfrac{c-c_0}{c_1-c_0}T\right) \right)  \notag \\
&=\left| \dfrac{c-c_0^{\prime}}{c_1^{\prime}-c_0^{\prime}}T -\dfrac{c-c_0}{c_1-c_0}T\right|
\end{align*}
By the triangle inequality $ (c_1-c_0)t  \leq (c_1^{\prime}-c_0^{\prime})s+ d_*(a,b)$ and inequality \eqref{key}, we have
\begin{align*}
\dfrac{c-c_0^{\prime}}{c_1^{\prime}-c_0^{\prime}}T -\dfrac{c-c_0}{c_1-c_0}T 
&=\dfrac{c-c_0^{\prime}}{c_1^{\prime}-c_0^{\prime}} (c_1-c_0)t  -\dfrac{c-c_0}{c_1-c_0}(c_1-c_0)t   \\
&\leq \dfrac{c-c_0^{\prime}}{c_1^{\prime}-c_0^{\prime}}  \{ (c_1^{\prime}-c_0^{\prime})s+ d_*(a,b)\} -(c-c_0)t   \\
&\leq (c-c_0^{\prime})s -(c-c_0)t +\dfrac{c-c_0^{\prime}}{c_1^{\prime}-c_0^{\prime}} d_*(a,b) \\
&\leq c(s-t) + cd_*(\Gamma_1(t),\Gamma_2(s)) \\
&\leq c \cdot \{2d_*(\Gamma_1(t), \Gamma_2(s))\}.
\end{align*}
Similarly, by the triangle inequality $ (c_1^{\prime}-c_0^{\prime})s - d_*(a,b)  \leq  (c_1-c_0)t$ and inequality \eqref{key}, we have
\begin{align*}
\dfrac{c-c_0}{c_1-c_0}T  - \dfrac{c-c_0^{\prime}}{c_1^{\prime}-c_0^{\prime}}T 
&=\dfrac{c-c_0}{c_1-c_0}(c_1-c_0)t  - \dfrac{c-c_0^{\prime}}{c_1^{\prime}-c_0^{\prime}} (c_1-c_0)t  \\
&\leq (c-c_0)t -\dfrac{c-c_0^{\prime}}{c_1^{\prime}-c_0^{\prime}}  \{ (c_1^{\prime}-c_0^{\prime})s - d_*(a,b)\} \\
&= (c-c_0)t - (c-c_0^{\prime})s +\dfrac{c-c_0^{\prime}}{c_1^{\prime}-c_0^{\prime}} d_*(a,b) \\
&\leq c(t-s) + cd_*(\Gamma_1(t),\Gamma_2(s)) \\
&\leq c \cdot \{2d_*(\Gamma_1(t), \Gamma_2(s))\}.
\end{align*}
Then we have
\begin{align}
d_*(a^{\prime}, \Gamma_1(ct))
&=\left| \dfrac{c-c_0^{\prime}}{c_1^{\prime}-c_0^{\prime}}T -\dfrac{c-c_0}{c_1-c_0}T\right| \notag \\
&\leq c \cdot \{2d_*(\Gamma_1(t), \Gamma_2(s))\}. \label{2}
\end{align}
Combining \eqref{1} and \eqref{2} yields
\begin{align*}
d_*(\Gamma_1(ct),\Gamma_2(cs)) &\leq d_*(\Gamma_1(ct),a^{\prime}) + d_*(a^{\prime}, \Gamma_2(cs)) \\ 
                                              & \leq c \cdot \{2d_*(\Gamma_1(t), \Gamma_2(s))\} + c \cdot Ed_*(\Gamma_1(t), \Gamma_2(s)) +C \\
                                              & = c(2+E) d_*(\Gamma_1(t), \Gamma_2(s)) +C.
\end{align*}
\par 
We consider case III).
In case III), we supposed that $\Gamma_1(ct) \in [p,a]$ and $\Gamma_2(cs) \in [b, \Gamma_2(s)]$,
where $a=\Gamma_1(c_1t)$ and $b=\Gamma_2(c^{\prime}_1s)$.
Let  $t^{\prime} \coloneqq c_1t$ and $s^{\prime} \coloneqq c^{\prime}_1s$.
We define
\begin{gather*}
a^{\prime} \coloneqq \Gamma_1(ct^{\prime}), \\ 
b^{\prime} \coloneqq \Gamma_2(cs^{\prime}).
\end{gather*}
\begin{figure}[h]
\begin{tikzpicture}
\draw(0,0)--++(2,1)--++(0,-2)--cycle;
\draw(-2,0)--++(2,0);
\draw(-2,0)node[left]{$\Gamma_1(0)$};
\draw(0,0)node[above]{$p$};
\fill(-2,0)circle(0.06);
\draw(2,1)--++(1,1.5);
\draw(2,1)node[above]{$a$};
\draw(3,2.5)node[above]{$\Gamma_1(t)$};
\fill(3,2.5)circle(0.06);
\draw(2,-1)--++(1,-1.5);
\draw(2,-1)node[right]{$b$};
\draw(3,-2.5)node[below]{$\Gamma_2(s)$};
\fill(3,-2.5)circle(0.06);
\draw(0.8,0.5)node[above]{$\Gamma_1(ct)$};
\fill(1,0.5)circle(0.06);
\draw(0.75,0.35)node[below]{$a^{\prime}$};
\fill(0.5,0.25)circle(0.06);
\draw(1,-0.8)node{$b^{\prime}$};
\fill(1,-0.5)circle(0.06);
\draw(3,-1.5)node{$\Gamma_2(cs)$};
\fill(2.5,-1.75)circle(0.06);
\end{tikzpicture}
\caption{Case III)}
\label{II)}
\end{figure}
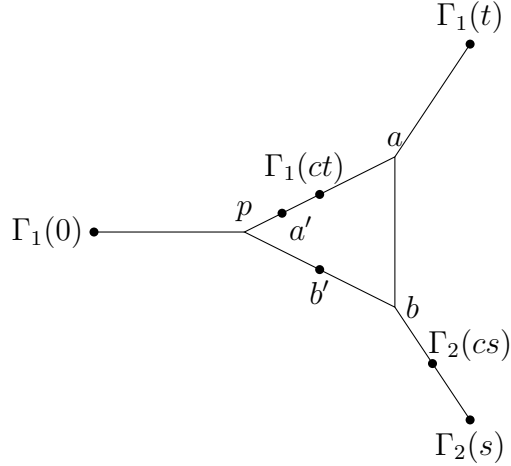
Since $a=\Gamma_1(t^{\prime})$ and $b=\Gamma_2(s^{\prime})$, we have $a^{\prime} \in [\Gamma_1(0), p] \cup [p,a]$ and $b^{\prime} \in [\Gamma_1(0), p] \cup [p,b]$.
By the same argument as in case I) or case II), we have
\begin{align*}
d_*(a^{\prime}, b^{\prime})
&=d_*(\Gamma_1(ct^{\prime}), \Gamma_2(cs^{\prime}) ) \\
&\leq  c(2+E) d_*(\Gamma_1(t^{\prime}), \Gamma_2(s^{\prime})) + C \\
&=c(2+E) d_*(a,b) +C .
\end{align*}
We remark that $d_*(\Gamma_1(t),\Gamma_2(s)) = d_*(\Gamma_1(t),a)+d_*(a,b)+d_*(b,\Gamma_2(s))$.
Then we have
\begin{align*}
&d_*(\Gamma_1(ct), \Gamma_2(cs)) \\
& \leq d_*(\Gamma_1(ct), a^{\prime}) + d_*(a^{\prime},b^{\prime}) +d_*(b^{\prime}, \Gamma_2(cs)) \\
& = d_*(\Gamma_1(ct), \Gamma_1(ct^{\prime}) )+ d_*(\Gamma_1(ct^{\prime}),\Gamma_2(cs^{\prime}))+d_*(\Gamma_2(cs^{\prime} ), \Gamma_2(cs)) \\
&\leq c(t-t^{\prime}) + c(2+E) d_*(\Gamma_1(t^{\prime}), \Gamma_2(s^{\prime})) + C + c(s-s^{\prime}) \\
&= cd_*(\Gamma_1(t),\Gamma_1(t^{\prime})) +c(2+E) d_*(\Gamma_1(t^{\prime}), \Gamma_2(s^{\prime})) + C + cd_*(\Gamma_2(s^{\prime}), \Gamma_2(s)) \\
&\leq c(2+E) \{ d_*(\Gamma_1(t), a) + d_*(a,b) + d_*(b, \Gamma_2(s)) \} +C \\
&= c(2+E) d_*(\Gamma_1(t),\Gamma_2(s)) + C. 
\end{align*}
\par 
We consider case IV).
In case IV), we supposed that $\Gamma_1(ct) \in [a, \Gamma_1(t)]$ and $\Gamma_2(cs) \in [b, \Gamma_2(s)]$,
where $a=\Gamma_1(c_1t)$, $b=\Gamma_2(c^{\prime}_1s)$,
and $c_1 \leq c^{\prime}_1$. 
We remark that $c_1 \leq c^{\prime}_1 \leq c$.
\begin{figure}[h]
\begin{tikzpicture}
\draw(0,0)--++(2,1)--++(0,-2)--cycle;
\draw(-2,0)--++(2,0);
\draw(-2,0)node[left]{$\Gamma_1(0)$};
\draw(0,0)node[above]{$p$};
\fill(-2,0)circle(0.06);
\draw(2,1)--++(1,1.5);
\draw(2,1)node[above]{$a$};
\draw(3,2.5)node[above]{$\Gamma_1(t)$};
\fill(3,2.5)circle(0.06);
\draw(2,-1)--++(1,-1.5);
\draw(2,-1)node[right]{$b$};
\draw(3,-2.5)node[below]{$\Gamma_2(s)$};
\fill(3,-2.5)circle(0.06);
\draw(2.5,1.5)node[right]{$\Gamma_1(ct)$};
\fill(2.5,1.75)circle(0.06);
\draw(3,-1.5)node{$\Gamma_2(cs)$};
\fill(2.5,-1.75)circle(0.06);
\end{tikzpicture}
\caption{Case IV)}
\label{III)}
\end{figure}
\vspace{0.5cm}
Then, by inequality \eqref{key}, we have
\begin{align*}
&d_*(\Gamma_1(ct),\Gamma_2(cs))  \\
&=d_*(\Gamma_1(ct),a) + d_*(a,b) + d_*(b, \Gamma_2(cs)) \\
&=d_*(\Gamma_1(ct),\Gamma_1(c_1t)) + d_*(a,b)+ d_*(\Gamma_2(c^{\prime}_1s),\Gamma_2(cs)) \\
&\leq c\left(1-\dfrac{c_1}{c} \right)t + c^{\prime}_1d_*(\Gamma_1(t),\Gamma_2(s)) + c\left( 1-\dfrac{c^{\prime}_1}{c} \right)s \\
&\leq c(1-c_1)t + cd_*(\Gamma_1(t),\Gamma_2(s))+ c(1-c^{\prime}_1)s \\
&= c \{ d_*(\Gamma_1(t),a) + d_*(b,\Gamma_2(s)) \} + cd_*(\Gamma_1(t),\Gamma_2(s)) \\
&\leq c \{2 d_*(\Gamma_1(t),\Gamma_2(s)) \}.
\end{align*}
Therefore, for all $\Gamma_1, \Gamma_2 \in \mathcal{L}_{*}$ with $\Gamma_1:[0,t] \to X*Y$ and $\Gamma_2 :[0,s] \to X*Y$,
\begin{equation}\label{diff}
d_*(\Gamma_1(ct), \Gamma_2(cs)) \leq c(2+E) d_*(\Gamma_1(t),\Gamma_2(s)) + C
\end{equation}
holds for all $c \in [0,1]$.
\par Finally, we will show that there exist $E_{**}\geq 0$ and $C_{**}\geq 0$ such that for all $t^{\prime} \in [0,t]$, $s^{\prime} \in [0,s]$, and $c \in [0,1]$, 
\begin{equation*}
d_*(\Gamma_1(ct^{\prime}), \Gamma_2(cs^{\prime})) \leq cE_{**} d_*(\Gamma_1(t^{\prime}), \Gamma_2(s^{\prime})) + C_{**}.
\end{equation*}
We recall that in general, $\Gamma_1|_{[0,t^{\prime}]}$ is not in $\mathcal{L}_{*}$.
\par
Let $a^{\prime}=\Gamma_1(t^{\prime})$ and $S' \in \mathcal{S}$ such that $a' \in S'$. 
Let $\Gamma_1' \in \mathcal{L}_{*}$ be the geodesic segment from $\Gamma_1(0)$ to $a^{\prime}$.
Since $\Gamma_1$ and $\Gamma^{\prime}_1$ are geodesic segments, we can define $\Gamma^{\prime}_1$ to be $\Gamma^{\prime}_1 : [0,t^{\prime}] \to X*Y$.
Let
\[
T\coloneqq \max \{ t \in [0,t'] \mid \Gamma_1(s)=\Gamma_1'(s) \ \text{for\ any}\ s \in [0,t] \}
\]
We can put $\Gamma_1(T)=\Gamma_1(c_0t')=\Gamma'_1(c_0t')$ for some $c_0 \in [0,1]$. 
Let $p=\Gamma_1(T)$. 
By the definition of $\mathcal{L}_*$, 
we obtain that $p$ is in $S'$.
Moreover, there exist $p^{\prime} \in S'$ satisfying the following conditions:
Let $\gamma_{S'}(p,p^{\prime}) \in \mathcal{L}$ be the geodesic segment 
with $\gamma_{S'}(p,p^{\prime}) : [0,u] \to S' $
and let $\gamma_{S'}(p,a^{\prime}) \in \mathcal{L}$ be the geodesic segment 
with $\gamma_{S'}(p,a^{\prime}) :[0,v] \to S'$.
Such that
\begin{align*}
\Gamma_1([c_0t^{\prime},t^{\prime}]) &\subseteq \gamma_{S'}(p,p^{\prime})([0,u]), \\
\Gamma^{\prime}_1([c_0t^{\prime},t^{\prime}]) &= \gamma_{S'}(p,a^{\prime})([0,v]).
\end{align*}
\par 
Note that $\Gamma_1([0,c_0t^{\prime}]) = \Gamma^{\prime}_1([0,c_0t^{\prime}])$, $u \geq v$, and $a^{\prime}=\Gamma_1(t^{\prime})=\gamma_{S'}(p,p^{\prime})(v)$.
Then for all $c \leq c_0$, we have that $\Gamma_1(ct^{\prime})=\Gamma^{\prime}_1(ct^{\prime})$.
We suppose that $c > c_0$.
Then there exist $c_1 \in [0,1]$ such that
\begin{align*}
\Gamma_1(ct^{\prime})&= \gamma_{S'}(p,p^{\prime})(c_1v), \\
\Gamma^{\prime}_1(ct^{\prime})&=\gamma_{S'}(p,a^{\prime})(c_1v).
\end{align*}
Since $X$ and $Y$ are symmetric geodesic coarsely convex spaces, by the coarsely convex inequality, we have that
\begin{align}
d_*(\Gamma_1(ct^{\prime}), \Gamma^{\prime}_1(ct^{\prime})) &= d_*(\gamma_{S'}(p,p^{\prime})(c_1v),\gamma_{S'}(p,a^{\prime})(c_1v)) \notag \\
                                                                                 &\leq c_1Ed_*(\gamma_{S'}(p,p^{\prime})(v),\gamma_{S'}(p,a^{\prime})(v))+C \notag \\
                                                                                &=c_1Ed_*(a^{\prime},a^{\prime}) + C=C.  \label{1t}
\end{align}
Set $b^{\prime}=\Gamma_2(s^{\prime})$.
We remark that $a=\Gamma_2(0)$.
Let $\Gamma(a,b^{\prime}) \in \mathcal{L}_{*}$ be the geodesic segment from $a$ to $b^{\prime}$ with $\Gamma(a,b^{\prime}) : [0,s^{\prime}] \to X*Y$.
For simplicity, we write $\Gamma(a,b^{\prime})$ as $\Gamma^{\prime}_2$.
By the same argument, we have that for all $c \in [0,1]$,
\begin{equation}\label{2s}
d_*(\Gamma_2(cs^{\prime}), \Gamma^{\prime}_2(cs^{\prime})) \leq C.
\end{equation}
Since $\Gamma^{\prime}_1$, $\Gamma^{\prime}_2 \in \mathcal{L}_{*}$ with $\Gamma^{\prime}_1 :[0,t^{\prime}] \to X*Y$ and $\Gamma^{\prime}_2 : [0,s^{\prime}] \to X*Y$, by \eqref{diff}, \eqref{1t}, and \eqref{2s},
\begin{align*}
&d_*(\Gamma_1(ct^{\prime}), \Gamma_2(cs^{\prime})) \\
&\leq d_*(\Gamma_1(ct^{\prime}),\Gamma^{\prime}_1(ct^{\prime}))+d_*(\Gamma^{\prime}_1(ct^{\prime}), \Gamma^{\prime}_2(cs^{\prime})) +d_*(\Gamma_2^{\prime}(cs^{\prime}), \Gamma_2(cs^{\prime})) \\
& \leq c(2+E) d_*(\Gamma^{\prime}_1(t^{\prime}),\Gamma^{\prime}_2(s^{\prime})) + 3C \\
&=c(2+E) d_*(\Gamma_1(t^{\prime}),\Gamma_2(s^{\prime}))+3C.
\end{align*}
Therefore, we obtain that (sgCC2) holds.
This completes the proof.
\end{proof}
\begin{proof}[Proof of Theorem~\ref{Ma}]
By the construction, the family of geodesic segments 
$\mathcal{L}_{*}$ is symmetric.
Therefore, by Proposition~\ref{sgCC} and~\ref{Propo L}, the free product
$X*Y$ is a symmetric geodesic coarsely convex space, in particular, $X*Y$ is a $(1,0,2+E,9C,\mathcal{L}_{*},\mathrm{id}_{\mathbb{R}_{\geq 0}})$-coarsely convex
space.
\end{proof}

\section{Group actions on free products of metric spaces}\label{S4}
Let $(X,d_X,e_X)$ and $(Y,d_Y,e_Y)$ be metric spaces with base points $e_X$ and $e_Y$, respectively.
Let $G$ and $H$ be groups acting properly and cocompactly on $X$ and $Y$, respectively.
Bridson and Haefliger \cite[Theorem II.11.18]{BH} construct a metric space $\overline{Z}$ 
on which the free product $G*H$ acts properly and cocompactly.
Moreover, when $X$ and $Y$ are CAT(0) spaces, they showed that $\overline{Z}$ is a CAT(0) space. 
In this section, we will show that $\overline{Z}$ is isometric to the free product $X*Y$ with respect to the $G$-net $(G, o(e_X),e_X)$ and the $H$-net $(H, o(e_Y),e_Y)$, where $o(e_X)$ and $o(e_Y)$ are the orbit maps.
\par
First, we briefly review their construction.
Let $\Gamma = G*H$.
Define $Z$ by 
\begin{equation*}
Z \coloneqq (\Gamma \times X) \bigsqcup (\Gamma \times [0,1]) \bigsqcup (\Gamma \times Y). 
\end{equation*}
Let $\overline{Z}$ be the quotient of $Z$ by the equivalent relation generated by:
\begin{align*}
(\omega g, x) &\sim (\omega, g(x)), (\omega h,y) \sim (\omega, h(y)), \\
(\omega, e_X) &\sim (\omega, 0), (\omega, e_Y) \sim (\omega,1)
\end{align*}
for all $\omega \in G*H$, $g \in G$, $h \in H$, $x \in X$, and $y \in Y$.
Let $\overline{X}$ be the quotient of $\Gamma \times X$ by the restriction of the above relation,
and let $\overline{Y}$ be the quotient of $\Gamma \times Y$ by the restriction of the above relation.
We remark that $\overline{X}$ is isometric to $(W_X \cup \{\epsilon\}) \times X$, where $W_X$ is the set of words of $G*H$ such that the last letter of each word of $W_X$ is in $H$,
and $\overline{Y}$ is isometric to $(W_Y \cup \{\epsilon\}) \times Y$, where $W_Y$ is the set of words of $G*H$ such that the last letter of each word of $W_Y$ is in $G$.
\par Let $\omega =u gh $, where $u \in G*H$, $g \in G$, and $h \in H$.
By the above equivalent relation, we have $(\omega, e_X) \sim (\omega, 0)$
and we have
\begin{align*}
(\omega, 1) &\sim (\omega, e_Y), \\
                 &=(ugh, e_Y), \\
                 &\sim (ug, h(e_Y)).
\end{align*}
Let $\tau =v ghg^{\prime} $, where $v \in G*H$, $g, g^{\prime} \in G$, and $h \in H$.
By the above equivalent relation, we have $(\tau, e_Y) \sim (\tau, 1)$
and we have
\begin{align*}
(\tau, 0) &\sim (\tau, e_X), \\
                 &=(ughg^{\prime}, e_X), \\
                 &\sim (ugh, g^{\prime}(e_X)).
\end{align*}
Note that $(\epsilon, e_X) \sim (\epsilon, 0)$ and $(\epsilon, 1) \sim (\epsilon, e_Y)$.
Therefore, $\overline{Z}$ consists of the following two types of components.
\begin{itemize}
\item The sheets consist of $\{\omega\}\times X$ and $\{\tau\} \times Y$, where $\omega \in W_X \cup \{\epsilon\}$ and $\tau \in W_Y \cup \{\epsilon\}$.
\item The edges consist of the following three:
\begin{itemize}
\item There exists an edge $\{\epsilon\} \times [0,1]$ connecting $(\epsilon, e_X) \in \{\epsilon\} \times X$ and $(\epsilon, e_Y) \in \{\epsilon\} \times Y$
\item Let $\omega \in W_X$ and $g \in G$.
        Then there exists an edge $\{\omega\} \times \{g\} \times [0,1]$ connecting $(\omega, g(e_X)) \in \{\omega\} \times X$ and $(\omega g, e_Y) \in \{\omega g\} \times Y$.
\item Let $\tau \in W_Y$ and $h \in H$. 
        Then there exists an edge $\{\tau\} \times \{h\} \times [0,1]$ connecting $(\tau, h(e_Y)) \in \{\tau\} \times Y$ and $(\tau h, e_X) \in \{\tau h\} \times X$.
\end{itemize}
\end{itemize}
Bridson and Haefliger \cite[Theorem II.11.18]{BH} showed that $G*H$ acts properly and cocompactly on $\overline{Z}$.
\par
We can easily show that $\overline{Z}$ is isometric to the free product $X*Y$ with respect to the $G$-net $(G, o(e_X),e_X)$ and the $H$-net $(H, o(e_Y),e_Y)$, where $o(e_X)$ and $o(e_Y)$ are the orbit maps.
Therefore, we have
\begin{proposition}[{\cite[Theorem II.11.18]{BH}}]
\label{prop:GHactsonXY}
Let $X$ and $Y$ be metric spaces.
Let $G$ and $H$ be groups acting properly and cocompactly on $X$ and $Y$, respectively.
We associate $X$ and $Y$ with the $G$-net and the $H$-net, respectively.
Then $G*H$ acts properly and cocompactly on $X*Y$.
\end{proposition}

\begin{corollary}\label{CCgp}
Let $X$ and $Y$ be symmetric geodesic coarsely convex spaces.
Let $G$ and $H$ be groups acting properly and cocompactly on $X$ and $Y$, respectively.
Then $G*H$ 
is a coarsely convex group.
\end{corollary}

\section{Application to the coarse Baum--Connes conjecture}
\subsection{Review of the coarse Baum--Connes conjecture}
Let $Y$ be a proper metric space.
We say that $Y$ satisfies the coarse Baum--Connes conjecture if the
following coarse
assembly map $\mu_Y$ of $Y$ is an isomorphism:
\[
 \mu_Y \colon KX_*(Y) \rightarrow K_*(C^*(Y)).
\]
Here, the left-hand side is the \textit{coarse K-homology} of $Y$, and
the right-hand side is the $K$-theory of 
the $C^*$-algebra $C^*(Y)$, called the \textit{Roe algebra} of $Y$.
Both are invariant under coarse equivalence, and the coarse assembly 
map behaves naturally for coarse maps. Therefore we have the following.
\begin{proposition}\label{re1}
Let $X$ and $Y$ be proper metric spaces.
We suppose that $X$ and $Y$ are coarsely equivalent.
If $X$ satisfies the coarse Baum--Connes conjecture, then so does $Y$.
\end{proposition}
For details, see \cite{HR1}, \cite{HRb}, and \cite{Yu2}.
The first author and Oguni show the following. 
\begin{theorem}[{\cite[Theorem 1.3]{FO}}]\label{FO}
Let $X$ be a proper coarsely convex space.
Then $X$ satisfies the coarse Baum--Connes conjecture.
\end{theorem}

We will apply Theorem~\ref{FO} for free products of 
proper symmetric geodesic coarsely convex spaces.

\begin{lemma}\label{FM}
Let $X$ and $Y$ be proper metric spaces with nets.
Then the free product $X*Y$ is a proper metric space.
\end{lemma}
\begin{proof}
Let $(X,d_X)$ and $(Y,d_Y)$ be proper metric spaces .
Let $(X_0, i_X, e_X)$ and $(Y_0, i_Y, e_Y)$ be nets of $X$ and $Y$, respectively.
Set $a \in X*Y$ and $R\geq0$.
Let $n$ be the height of the sheet that contains $a \in X*Y$.
We define $B_*(a,R)$ to be
\begin{equation*}
 B_*(a,R) \coloneqq \{ b \in X*Y \mid d_*(a,b) \leq R \}.
\end{equation*}
By the definition of nets, 
the height of the sheets that intersect $B_*(a,R)$ is in the interval $\left[n- [R]-1 , n+ [R]+1\right]$.
Since $X$ and $Y$ are proper metric spaces, by the definition of nets,  for any bounded closed subset $K$ in each sheet, the preimage of $K$ by the index maps is a finite set.
Therefore, $B_*(a,R)$ is the finite union of bounded closed subsets in the sheets.
Since $X$ and $Y$ are proper metric spaces, $B_*(a,R)$ is compact.
\end{proof}
Combining Theorem \ref{Ma}, Theorem \ref{FO}, and Lemma \ref{FM}, we obtain the following.
\begin{theorem}\label{FM2}
Let $X$ and $Y$ be proper metric spaces with nets. 
If $X$ and $Y$ are symmetric geodesic coarsely convex spaces,
then the free product $X*Y$
satisfies the coarse Baum--Connes conjecture.
\end{theorem}
Let $G$ and $H$ be finitely generated groups acting properly and cocompactly on $X$ and $Y$, respectively.
As mentioned in Section \ref{S4}, the free product $G*H$ acts properly and cocompactly on $X*Y$ with respect to the $G$-net and the $H$-net.
Therefore, combining Theorem \ref{FM2}, the \v{S}varc--Milnor Lemma, and Proposition \ref{re1}, we obtain the following result.
\begin{theorem}\label{FM3}
Let $X$ and $Y$ be proper metric spaces with nets.
We suppose that $X$ and $Y$ are symmetric geodesic coarsely convex spaces.
Let $G$ and $H$ be groups acting properly and cocompactly on $X$ and $Y$, respectively.
Then the free product $G*H$ satisfies the coarse Baum--Connes conjecture.
\end{theorem}
Finally, we compare Theorem \ref{FM2} and Theorem \ref{FM3} with some known results
for relatively hyperbolic groups, and spaces admitting a coarse embedding 
into a Hilbert space.

\subsection{Relatively hyperbolic groups}
The first author and Oguni \cite{FOrel} showed the following.
\begin{theorem}[{\cite[Theorem 1.1]{FOrel}}]\label{fmrel}
Let $G$ be a finitely generated group and $\mathbb{P}=\{P_1,\dots,P_k\}$ be a finite family of infinite subgroups.
Suppose that $(G,\mathbb{P})$ is a relatively hyperbolic group.
If each subgroup $P_i$ satisfies the coarse Baum--Connes conjecture, 
and admits a finite $P_i$-simplicial complex which is a universal space for proper actions, 
then $G$ satisfies the coarse Baum--Connes conjecture.
\end{theorem}
Let $G$ and $H$ be finitely generated groups.
The free product $G*H$ is hyperbolic relative to $\{G,H\}$.
If $G$ and $H$ act properly and cocompactly on any spaces listed in Table \ref{Tab1},
then $G$ and $H$ admit finite $G$-simplicial (resp. $H$-simplicial) complexes which are universal spaces for proper actions.
Therefore, $G*H$ satisfies the assumptions of Theorem \ref{fmrel}.
However, it is not known in general whether groups acting properly and cocompactly on symmetric geodesic coarsely convex spaces 
always admit finite $G$-simplicial complexes which are universal spaces for proper actions.

\subsection{Spaces admitting a coarse embedding into a Hilbert space}
The notion of coarse embedding into a Hilbert space is introduced by
Gromov~\cite{asym_invs}.
Yu \cite{Yu1} showed the following.
\begin{theorem}[{\cite[Theorem 1.1]{Yu1}}]\label{YuH}
Let $\Gamma$ be a discrete metric space with bounded geometry.
If $\Gamma$ admits a uniform embedding into a Hilbert space,
then the coarse Baum--Connes conjecture holds for $\Gamma$.
\end{theorem}
Let $X$ and $Y$ be proper metric spaces.
Let $G$ and $H$ be groups acting properly and cocompactly on $X$ and $Y$, respectively.
We suppose that $X$ and $Y$ admit coarse embeddings into Hilbert spaces.
Then $G$ and $H$ also admit coarse embeddings into Hilbert spaces.
By the work of Dadarlat--Guentner \cite{DG}, the free product $G*H$ embeds coarsely into the Hilbert space.
By Theorem \ref{YuH}, the free product $G*H$ satisfies the coarse Baum--Connes conjecture.
We remark that $G*H$ acts properly and cocompactly on $X*Y$.
Then, by the \v{S}varc--Milnor Lemma and Proposition \ref{re1}, 
it follows that $X*Y$ satisfies the coarse Baum--Connes conjecture.
\par 
We remark that in the above setting, all $X$, $Y$, and $X*Y$ are with
bounded coarse geometry in the sense of \cite[Definition A.2]{Busemann_cBC}.
However, there exist symmetric geodesic coarsely
convex spaces without bounded coarse geometry.

\begin{example}
 Let $\Gamma$ be the Cayley graph of $\Z/2\Z * \Z/3\Z$ for some
 generating set.  Since $\Z/2\Z * \Z/3\Z$ is a hyperbolic group,
 $\Gamma$ is a symmetric geodesic coarsely convex space.

 For $p\in (0,\infty)$, let $X_p$ be the proper Busemann space given in
 \cite[Example 2.2]{Busemann_cBC}.
 As described in \cite{Busemann_cBC}, $X_p$ is constructed
 from the half-line $[0,\infty)$ by identifying each integer 
 $n\in [0,\infty)$ with the origin of the $n$-dimensional $l_p$ space.
 By Theorem~\ref{FM2}, the free product $\Gamma * X_p$ satisfies 
 the coarse Baum--Connes conjecture.
 
 In \cite[Appendix]{Busemann_cBC}, it is shown that
 $X_p$ is without bounded coarse geometry. 
 Thus $\Gamma * X_p$ does not satisfy the assumptions of Theorem~\ref{YuH}.
 
 It also follows that $X_p$ does not admit any proper cocompact actions by
 discrete groups.  Therefore we cannot apply Theorem~\ref{fmrel} 
 to $\Gamma * X_p$.
\end{example}

Using expander graphs, Kondo \cite{Kondo} constructed a CAT(0) space
which is not coarsely embeddable into a Hilbert space. The space given in
\cite{Kondo} is not proper, however, with a slight modification, we
obtain a proper CAT(0) space which is not coarsely embeddable into
a Hilbert space. Following Kondo's argument \cite{Kondo}, we give the
construction of the space.

First, we briefly review expander graphs.
Let $G=(V,E)$ be a graph.
For $A \subset V$, we denote by $\partial^e A$, the set of edges connecting a vertex in $A$ and a vertex in $A^c$, that is,
\[ \partial^e A \coloneqq \{ e \in E : \#(e \cap A)=1\}.\]
The \textit{Cheeger constant} $h(G)$ is defined by
\[ h(G) \coloneqq \min \left\{ \frac{\# \partial^e A}{\# A} : A \subset G, 0 < \# A \leq \frac{\# V}{2} \right\}.\]
The \textit{girth} of $G$ is the length of the shortest embedded cycle contained in $G$, denoted by $\text{girth}(G)$.
\par A family of  \textit{expander graphs} is a sequence of finite connected graphs $\{ G_n=(V_n,E_n) \}_{n=1}^{\infty}$ satisfying the following conditions:
\begin{enumerate}
\item $\# V_n \to \infty \quad (n \to \infty).$
\item There exists $k \in \mathbb{N}$ such that for any $n \in \mathbb{N}$ and $v \in V_n$, the degree of $v$ is less than or equal to $k$.
\item There exists $c >0$ such that $h(G_n) >c$ for any $n \in \mathbb{N}$.
\end{enumerate}
We remark that a family of expander graphs is not coarsely embeddable
into Hilbert space.

Let a sequence of finite connected $k$-regular graphs $\{ G_n=(V_n,E_n) \}_{n=1}^{\infty}$ 
form a family of expander graphs satisfying $\text{girth}(G_n) \to \infty (n \to \infty)$
while 
\[
\frac{\mathrm{diam}(G_n)}{\mathrm{girth}(G_n)}
\]
remains bounded. 
Let $\rho_n$ be the combinatorial distance on $G_n$ and let 
\[ d_n \coloneqq \frac{ 2 \pi} { \text{girth} (G_n)}\rho_n. \]
Then each $(G_n, d_n)$ is a CAT(1) space since we are considering the scaled distance $d_n$ on $G_n$ and the length of the shortest embedded circle in it is 2$\pi$.
\par For each $G_n$, let $\widetilde{G_n} \coloneqq G_n \times \mathbb{R}_{\geq 0} / G_n \times \{0\}$. 
We construct a metric space $Y_0$ from the half-line $[0,\infty)$ by identifying each integer $n\in [0,\infty)$ with $G_n \times \{0\}$ of $\widetilde{G_n}$
and the distance of $Y_0$ is defined by 
\[ d_{Y_0}((v_1,r_1),(v_2,r_2))^2 \coloneqq r_1^2 + r_2^2 - 2 r_1 r_2 \cos (\min \{d_n(v_1,v_2), \pi\}),\]
when $v_1$ and $v_2$ are in the same $G_n$, and
\[ d_{Y_0}((v_1,r_1),(v_2,r_2))=r_1+r_2+|m-n|,\]
when $v_1 \in G_n, v_2 \in G_m$ for $n \neq m$.
It is clear that $(Y_0, d_{Y_0})$ is proper.
For each $n \in \mathbb{N}$, the metric space $(\widetilde{G_n}, d_{Y_0} |_{\widetilde{G_n}})$ is the Euclidean cone over $G_n$.
Since $G_n$ is a CAT(1) space, by Berestovskiĭ's theorem \cite{Eucc} (see also \cite[Theorem II.3.14]{BH}), 
the Euclidean cone $(\widetilde{G_n}, d_{Y_0} |_{\widetilde{G_n}})$ is a CAT(0) space.
Therefore, by using the gluing lemma \cite[Theorem II.11.3]{BH} repeatedly, it follows that $(Y_0,d_{Y_0})$ is a proper CAT(0) space. 
Since $(Y_0,d_{Y_0})$ contains a bi-Lipschitz embedded family of expanders
\cite[Proposition 4.4 and Remark 4.6]{Kondo}, 
this space is not coarsely embeddable into a Hilbert space.

\begin{remark}
Such a family of expanders is obtained from the Ramanujan graphs constructed by Lubotzky et al. \cite{Lu}.
\end{remark}

\begin{example}\label{Kondo}
Let $\Gamma$ be the Cayley graph of $\Z/2\Z * \Z/3\Z$ for some
 generating set.  Since $\Z/2\Z * \Z/3\Z$ is a hyperbolic group,
 $\Gamma$ is a symmetric geodesic coarsely convex space.
 We remark that $\Gamma$ is not a CAT(0) space.
 By Theorem~\ref{FM2}, the free product $\Gamma * Y_0$ satisfies 
 the coarse Baum--Connes conjecture.
 \par The free product $\Gamma * Y_0$ is neither CAT(0) nor 
 coarsely embeddable into any Hilbert space.
 Thus, $\Gamma * Y_0$ does not satisfy the assumptions of Theorem~\ref{YuH}.
\end{example}

\section*{Acknowlegement}
We appreciate the referee for his or her careful reading of the manuscript and valuable comments.
Our manuscript has been improved by his or her comments.
We would like to thank Professor Shinichi Oguni for very helpful discussion.
We would also like to thank Professor Takashi Shioya for very helpful comments on shapes of free products,
which led us to Example \ref{eg:finite_net}.
We are deeply grateful to Professor Takefumi Kondo for suggesting the construction of the proper CAT(0) space in Example \ref{Kondo}.
\par
T. Fukaya was supported by JSPS KAKENHI Grant number JP19K03471. T. Matsuka was supported by JST, the establishment of university fellowships towards the creation of science technology innovations, Grant number JPMJFS2139.

\bibliographystyle{plain}
\bibliography{myrefs}

\begin{thebibliography}{10}

\bibitem{Eucc}
V.~N. Berestovski\u{\i}.
\newblock Borsuk's problem on metrization of a polyhedron.
\newblock {\em Dokl. Akad. Nauk SSSR}, 268(2):273--277, 1983.

\bibitem{BH}
Martin~R. Bridson and Andr\'{e} Haefliger.
\newblock {\em Metric spaces of non-positive curvature}, volume 319 of {\em
  Grundlehren der mathematischen Wissenschaften [Fundamental Principles of
  Mathematical Sciences]}.
\newblock Springer-Verlag, Berlin, 1999.

\bibitem{Helly}
J{\'e}r{\'e}mie {Chalopin}, Victor {Chepoi}, Anthony {Genevois}, Hiroshi
  {Hirai}, and Damian {Osajda}.
\newblock {Helly groups}.
\newblock {\em arXiv e-prints}, page arXiv:2002.06895, February 2020.

\bibitem{DG}
Marius Dadarlat and Erik Guentner.
\newblock Uniform embeddability of relatively hyperbolic groups.
\newblock {\em J. Reine Angew. Math.}, 612:1--15, 2007.

\bibitem{DL}
Dominic Descombes and Urs Lang.
\newblock Convex geodesic bicombings and hyperbolicity.
\newblock {\em Geom. Dedicata}, 177:367--384, 2015.

\bibitem{FOrel}
Tomohiro Fukaya and Shin-ichi Oguni.
\newblock The coarse {B}aum-{C}onnes conjecture for relatively hyperbolic
  groups.
\newblock {\em J. Topol. Anal.}, 4(1):99--113, 2012.

\bibitem{Busemann_cBC}
Tomohiro Fukaya and Shin-ichi Oguni.
\newblock The coarse {B}aum--{C}onnes conjecture for {B}usemann nonpositively
  curved spaces.
\newblock {\em Kyoto J. Math.}, 56(1):1--12, 2016.

\bibitem{FO}
Tomohiro Fukaya and Shin-ichi Oguni.
\newblock A coarse {C}artan-{H}adamard theorem with application to the coarse
  {B}aum-{C}onnes conjecture.
\newblock {\em J. Topol. Anal.}, 12(3):857--895, 2020.

\bibitem{asym_invs}
M.~Gromov.
\newblock Asymptotic invariants of infinite groups.
\newblock In {\em Geometric group theory, {V}ol.\ 2 ({S}ussex, 1991)}, volume
  182 of {\em London Math. Soc. Lecture Note Ser.}, pages 1--295. Cambridge
  Univ. Press, Cambridge, 1993.

\bibitem{HR1}
Nigel Higson and John Roe.
\newblock On the coarse {B}aum-{C}onnes conjecture.
\newblock In {\em Novikov conjectures, index theorems and rigidity, {V}ol. 2
  ({O}berwolfach, 1993)}, volume 227 of {\em London Math. Soc. Lecture Note
  Ser.}, pages 227--254. Cambridge Univ. Press, Cambridge, 1995.

\bibitem{HRb}
Nigel Higson and John Roe.
\newblock {\em Analytic {$K$}-homology}.
\newblock Oxford Mathematical Monographs. Oxford University Press, Oxford,
  2000.
\newblock Oxford Science Publications.

\bibitem{Kondo}
Takefumi Kondo.
\newblock {${\rm CAT}(0)$} spaces and expanders.
\newblock {\em Math. Z.}, 271(1-2):343--355, 2012.

\bibitem{Lu}
A.~Lubotzky, R.~Phillips, and P.~Sarnak.
\newblock Ramanujan graphs.
\newblock {\em Combinatorica}, 8(3):261--277, 1988.

\bibitem{OP}
Damian Osajda and Piotr Przytycki.
\newblock Boundaries of systolic groups.
\newblock {\em Geom. Topol.}, 13(5):2807--2880, 2009.

\bibitem{PT}
Toma\v{z} Pisanski and Thomas~W. Tucker.
\newblock Growth in products of graphs.
\newblock {\em Australas. J. Combin.}, 26:155--169, 2002.

\bibitem{Yu2}
Guoliang Yu.
\newblock Coarse {B}aum-{C}onnes conjecture.
\newblock {\em $K$-Theory}, 9(3):199--221, 1995.

\bibitem{Yu1}
Guoliang Yu.
\newblock The coarse {B}aum-{C}onnes conjecture for spaces which admit a
  uniform embedding into {H}ilbert space.
\newblock {\em Invent. Math.}, 139(1):201--240, 2000.

\end{thebibliography}


\begin{thebibliography}{99}

\end{thebibliography}

\end{document}